\documentclass[leqno]{siamltex704}
\usepackage{amsmath}
\usepackage{graphicx}
\usepackage{mathrsfs}
\usepackage{rotating}
\usepackage{lscape}
\usepackage{float}
\usepackage{amsfonts,amssymb}
\usepackage{dsfont}
\usepackage{pifont}
\usepackage{tikz}
\usepackage{wrapfig} 
\usepackage{hyperref}
\usepackage{multirow}
\usepackage{bm}
\usepackage{mathtools}
\usepackage{cases}
\usepackage{subfigure}
\usepackage{colortbl}
\usepackage{enumitem}
\definecolor{mygray}{gray}{.9}
\definecolor{mypink}{rgb}{.99,.91,.95}
\definecolor{mycyan}{cmyk}{.3,0,0,0}
\newcommand{\vertiii}[1]{{\left\vert\kern-0.25ex\left\vert\kern-0.25ex\left\vert #1
    \right\vert\kern-0.25ex\right\vert\kern-0.25ex\right\vert}}
\newtheorem{remark}{Remark}[section]

\def\diam{\operatornamewithlimits{diam}}

\def\s{{\mathcal S}}
\def\T{{\mathcal T}}

\def\E{{\mathcal E}}

\def\bn{{\boldsymbol n}}
\def\bq{{\boldsymbol q}}

\def\bv{{\boldsymbol v}}
\def\pT{{\partial T}}

\newcommand{\bu}{{\bm u}}
\newcommand{\bV}{{\boldsymbol V}}

\newcommand\subsetsim{\mathrel{%
  \ooalign{\raise0.2ex\hbox{$\subset$}\cr\hidewidth\raise-0.8ex\hbox{\scalebox{0.9}{$\sim$}}\hidewidth\cr}}}
\def\leqC{\lesssim}

\def\T{{\mathcal T}}
\def\E{{\mathcal E}}

\def\functionalaction#1#2{{ (\!(#1, #2)\!)}}

\def\bn{{\boldsymbol n}}
\def\bq{{\boldsymbol q}}

\def\bH{{\bf H}}

\def\boldeta{{\boldsymbol\eta}}
\def\bpsi{{\boldsymbol\psi}}

\def\bvarphi{{\boldsymbol\varphi}}

\def\HnHarmonic{\mathbb{H}_{n,0}(\Omega)}

\def\bbQ{\mathbb{Q}}

\newcommand{\curl}{{\nabla\times}}
\newcommand{\PreserveBackslash}[1]{\let\temp=\\#1\let\\=\temp}
\newcolumntype{C}[1]{>{\PreserveBackslash\centering}p{#1}}
\newcolumntype{R}[1]{>{\PreserveBackslash\raggedleft}p{#1}}
\newcolumntype{L}[1]{>{\PreserveBackslash\raggedright}p{#1}}
\def\3bar{{|\!|\!|}}

\def\bx{{\boldsymbol{x}}}
\newtheorem{algorithm}{Algorithm}[section]

\title{A Primal-Dual Weak Galerkin Method for Div-Curl Systems with low-regularity solutions}

\author{Yujie Liu\thanks{Artificial Intelligence Research Center, Peng Cheng Laboratory, Shenzhen 518005, China (liuyj02@pcl.ac.cn). The research of Liu was partially supported by National Natural Science
Foundation of China (No. 12001306), Guangdong Provincial Natural Science Foundation (No. 2017A030310285), Shandong Provincial Natural Science Foundation (No. ZR2016AB15) and Youthful Teacher Foster Plan Of Sun Yat-Sen University (No. 171gpy118),} \and Junping Wang
\thanks{Division of Mathematical Sciences, National Science Foundation, Alexandria, VA 22314 (jwang@nsf.gov). The research of Wang was supported in part by the NSF IR/D program, while working at National Science Foundation. However, any opinion, finding, and conclusions or recommendations expressed in this material are those of the author and do not necessarily reflect the views of the National Science Foundation.}}

\begin{document}

\maketitle

\begin{abstract}
This article presents a new primal-dual weak Galerkin finite element method for the div-curl system with tangential boundary conditions and low-regularity assumptions on the solution. The numerical scheme is based on a weak variational form involving no partial derivatives of the exact solution supplemented by a dual or adjoint problem in the general context of the weak Galerkin finite element method. Optimal order error estimates in $L^2$ are established for solution vector fields in $H^\theta(\Omega),\ \theta>\frac12$. The mathematical theory was derived on connected domains with general topological properties (namely, arbitrary first and second Betti numbers). Numerical results are reported to confirm the theoretical convergence.
\end{abstract}

\begin{keywords}
primal-dual weak Galerkin, finite element methods, div-curl system, tangential boundary conditions, low-regularity.
\end{keywords}

\begin{AMS}
Primary 65N30, 65N12, 65N15; Secondary 35Q60, 35B45
\end{AMS}

\pagestyle{myheadings}

\section{Introduction}
This paper is concerned with the development of a primal-dual weak Galerkin finite element method for div-curl systems equipped with tangential boundary conditions. For simplicity, consider the problem of seeking a vector field $\bu$ satisfying
\begin{subnumcases}
~~~~~~~\nabla\cdot(\varepsilon \bm{u})=f,&${\rm in}\  \Omega$, \label{EQ:div-curl-1}\\
~~~~~~~~~\nabla\times \bm{u}= \bm{g},&${\rm in}\  \Omega$, \label{EQ:div-curl-2}\\
~~~~~~~~~\bm{u} \times \bm{n}=\bm{\chi},&${\rm on}\ \Gamma$, \label{EQ:div-curl-TBC1}\\
\langle \varepsilon \bm{u} \cdot \bm{n}_i,1\rangle_{\Gamma_i}=\alpha_i,&$i=1,...,L,$ \label{EQ:div-curl-TBC2}
\end{subnumcases}
where $\Omega$ is an open bounded and connected domain in ${\mathbb R}^3 $, with Lipschitz continuous boundary $\Gamma = \partial \Omega$ which consists of a finite number of disjoint surfaces $
\Gamma=\mathop\bigcup_{i=0}^{L}\Gamma_i$, where each component $\Gamma_i$ is connected, Lipschitz continuous, and has finite surface area.
$\Gamma_0$ is the exterior boundary of the domain, and each $\Gamma_i$ corresponds to a ``hole" so that $L$ geometrically counts the number of holes in the region $\Omega$. The number $L$ is known as the second Betti number of $\Omega$ or the dimension of the second de Rham cohomology group of $\Omega$. In the equation \eqref{EQ:div-curl-1},  $\varepsilon= \{\varepsilon_{ij}(x)\}_{3\times3}$ is a symmetric and uniformly positive definite matrix in $\Omega$ with entries in $L^{\infty}(\Omega)$. The load function $f = f(\bx)$ is Lebesgue-integrable and real-valued, and the vector field $\bm{g} = \bm{g}(\bx)$ are given in the domain $\Omega$. The tangential boundary condition \eqref{EQ:div-curl-TBC1} corresponds to a given value for the tangential component of the vector field $\bm{u}$, where $\bm{n}_i$ is the unit outward normal direction on $\Gamma_i$,
$\bm{\chi}\in \textcolor{blue}{[L^2(\Gamma)]^{3}}$ is a given vector field on the boundary. 

The div-curl system has many important applications in computational fluid and electromagnetic problems \cite{Bossavit_1998}. For this reason, various numerical methods have been developed for solving this system in the last several decades. A control volume method was proposed by Nicolaides \cite{Nicolaides_1992} for the planar div-curl problems in 1992, thereafter a co-finite volume method was developed by Nicolaides and Wu \cite{Nicolaides_1997} for three dimensional div-curl problems. Delcourte \textit{et al.} \cite{DelcourteDomelevoOmnes_2007} proposed a discrete duality finite volume method for the div-curl problems on almost arbitrary polygonal meshes. Bramble and Pasciak \cite{BramblePasciak_2003} developed a finite element formulation for the div-curl systems under a very weak formulation where the solution space was $[L_2(\Omega)]^3$. In \cite{CopelandGopalakrishnanPasciak_2008}, Copeland \textit{et al.} presented a mixed finite element method for 3D axisymmetric div-curl systems, which reduces the computational domain from 3D to 2D via cylindrical coordinates in simply connected axisymmetric domains. A least-squares method based on discontinuous elements was propose by Bensow and Larson in \cite{BochevPetersonSiefert_2011}.
Bochev \textit{et al.} \cite{Bensow2005} proposed a least-squares finite element methods for two div-curl elliptic boundary value problems. The mimetic finite difference method was proposed by Brezzi \textit{et al.} \cite{BrezziBuffa_2010, LipnikovManziniBrezzi_2011} and \textcolor{blue}{applied} to the 3D magnetostatic problems  on general polyhedral meshes. Most recently, Wang \textit{et al.} \cite{WangWang_divcurl_2016} developed a weak Galerkin method for the div-curl systems with either normal or tangential boundary conditions. Ye \textit{et al.} proposed least-squares methods \cite{LiYeZhang_2018,YeZhangZhu_2020} for the div-curl problems, which result in symmetric positive definite linear systems.

For the div-curl system \eqref{EQ:div-curl-1}-\eqref{EQ:div-curl-TBC2} to be well-posed, the functional data in the system must satisfy certain compatibility conditions (e.g., the equations \eqref{CompatibilityC:01}, \eqref{EQ:divcurl-compatibility:01}, and \eqref{EQ:divcurl-compatibility:03}). In fact, the tangential boundary value problem \eqref{EQ:div-curl-1}-\eqref{EQ:div-curl-TBC2} has one and only one solution if all the desired compatibility conditions are met. One of the main challenges in the design of numerical methods for \eqref{EQ:div-curl-1}-\eqref{EQ:div-curl-TBC2} is the low-regularity nature of the exact solution $\bu$. The goal of this paper is to address this challenge by devising a primal-dual weak Galerkin (PDWG) scheme which provides reliable numerical approximations for solutions with low-regularity. In particular, the new PDWG finite element method will approximate the vector field $\bu$ by using piecewise polynomials in $[L^2(\Omega)]^2$, and an optimal order error estimate shall be derived in $L^2$ for solutions in $H^\theta(\Omega), \ \theta>\frac12$. The primal-dual idea for solving PDEs was also developed by Burman \cite{burman2013, burman2014} in other finite element contexts. The PDWG method has been successfully applied to several challenging problems including the second order elliptic equation in non-divergence form \cite{WangWang_2016}, the Fokker-Planck equation \cite{WangWang_2017}, the elliptic Cauchy problem \cite{ww2020-ec}, and linear transport problems \cite{ww2020-transport}.

Throughout the paper, we follow the usual notation for Sobolev spaces and norms as in \cite{Ciarlet_1978, girault-raviart}. For any open bounded domain $D \subset {\mathbb R}^3$ with Lipschitz continuous boundary, we use $\|\cdot\|_{s,D}$ and $|\cdot|_{s,D}$ to denote the norm and seminorm in the Sobolev space $H^s(D)$ for any $s\geq 0$, respectively. The inner product in $H^s(D)$ is denoted by $(\cdot,\cdot)_{s,D}$.
The space $H^0(D)$ coincides with $L^2(D)$, for which the norm and the inner product are denoted by $\|\cdot\|_D$ and $(\cdot,\cdot)_{D}$, respectively.  We use $H(div_\varepsilon;D)$ to denote the closed subspace of $[L^2(D)]^3$ so that $\nabla\cdot(\varepsilon\bv)\in L^2(D)$. The space $H(div;D)$ corresponds to the case of the identity matrix $\varepsilon=I$. Analogously, we use $H(curl;D)$ to denote the closed subspace of $[L^2(D)]^3$ so that $\nabla\times\bv\in [L^2(D)]^3$. The space of normal $\varepsilon$-harmonic vector fields, denoted by $\mathbb{H}_{\varepsilon n,0}(\Omega)$, consists of all $\varepsilon$-harmonic vector fields satisfying the zero normal boundary condition; i.e.,
$$
\mathbb{H}_{\varepsilon n,0}(\Omega)=\{\bv\in [L^2(\Omega)]^3: \ \curl\bv=0,\
\nabla\cdot(\varepsilon \bv)=0,\ \varepsilon\bv\cdot \bn = 0 \mbox{ on } \Gamma\}.
$$
When $\varepsilon=I$ is the identity matrix, the spaces $\mathbb{H}_{\varepsilon n,0}(\Omega)$ shall be denoted as $\mathbb{H}_{n,0}(\Omega)$.
\section{A Weak Formulation}
Denote by $\bar{H}^1(\Omega)=H^1(\Omega)/\mathbb{R}$ the quotient space and
$$
H^1_{0c}(\Omega)=\{w\in H^1(\Omega):\ w|_{\Gamma_0}=0, \ w|_{\Gamma_i}=c_i,\ i=1,\cdots, L, \ c_i \in  \mathbb{R}\}
$$
the closed subspace of $H^1(\Omega)$ with vanishing value on $\Gamma_0$ and constant value on each $\Gamma_i, \ i=1,\cdots, L$.

A {\it weak formulation} for the tangential boundary value problem \eqref{EQ:div-curl-1}-\eqref{EQ:div-curl-TBC2} seeks  $\bu\in [L^2(\Omega)]^3$ and $s \in {\bar{H}^1(\Omega)}$ such that
\begin{equation}\label{EQ:divcurl-tbv-weakform}
b(\bu, s; \varphi,\bpsi) = G(\varphi,\bpsi),\quad \forall \varphi\in H^1_{0c}(\Omega),\ \bpsi\in H(curl;\Omega),
\end{equation}
where
\begin{eqnarray}\label{EQ:div-curl:tbvp:bform}
&&b(\bu, s; \varphi,\bpsi): = (\bu, \varepsilon \nabla \varphi + \nabla \times \bpsi) + (\bpsi, \nabla s),\\
&&G(\varphi,\bm{\psi}):=(\bm{g},\bpsi)-(f,\varphi)+
\langle\bm{\chi}\times\bn,\bpsi\times\bn\rangle +\sum_{i=1}^L\alpha_i \varphi|_{\Gamma_i}. \label{EQ:div-curl:tbvp:gform}
\end{eqnarray}
The corresponding homogeneous dual or adjoint problem for \eqref{EQ:divcurl-tbv-weakform} seeks $\lambda\in H^1_{0c}(\Omega)$ and $\bq \in H(curl;\Omega)$ such that
\begin{eqnarray}\label{EQ:div-curl-tbvp:dual-problem}
&&b(\bm{v}, r; \lambda,\bq)=0,\qquad \forall \bv\in [L^2(\Omega)]^3,\ r\in \bar{H}^1(\Omega).
\end{eqnarray}

For \eqref{EQ:div-curl-1}-\eqref{EQ:div-curl-TBC2} to be well-imposed, certain compatibility conditions must be satisfied for the boundary value and the load functions $f$ and $\bm{g}$.  First, from the curl equation \eqref{EQ:div-curl-2} we have
\begin{equation}\label{CompatibilityC:01}
\nabla\cdot\bm{g} = 0.
\end{equation}
Next, as $\bm{\chi}=\bm{u}\times{\bn}$ is orthogonal to the normal direction $\bn$, hence
\begin{equation}\label{CompatibilityC:02}
\bm{\chi} =  \bn\times (\bm{\chi}\times\bn).
\end{equation}
By testing the equation \eqref{EQ:div-curl-2} against any $\bm{\psi}\in H(curl;\Omega)$ we have from the Green's formula and the tangential boundary condition \eqref{EQ:div-curl-TBC1} that
\begin{eqnarray}\label{EQ:variational-form-2:02}
(\bm{u}, \nabla \times \bm{\psi}) = (\bm{g},\bm{\psi})+\langle\bm{\chi}, \bm{\psi}\rangle,\qquad \forall \;\bm{\psi}\in H(curl;\Omega).
\end{eqnarray}
By letting $\bm{\psi}=\nabla\rho$ in \eqref{EQ:variational-form-2:02} and then using \eqref{CompatibilityC:01}-\eqref{CompatibilityC:02}, we arrive at the following compatibility condition:
\begin{equation}\label{EQ:divcurl-compatibility:01}
\langle\bm{g}\cdot\bn,\rho\rangle+\langle\bm{\chi}\times\bn, \nabla\rho\times\bn\rangle = 0,\qquad \forall\; \rho\in H^1(\Omega).
\end{equation}
Analogously, by letting $\bm{\psi}=\bm{\eta} \in \mathbb{H}_{n,0}(\Omega)$ in \eqref{EQ:variational-form-2:02}, we have
\begin{equation}\label{EQ:divcurl-compatibility:02}
(\bm{g},\bm{\eta})+\langle\bm{\chi}, \bm{\eta}\rangle = 0,\qquad \forall\; \bm{\eta} \in \mathbb{H}_{n,0}(\Omega),
\end{equation}
which, together with \eqref{CompatibilityC:02}, leads to
\begin{equation}\label{EQ:divcurl-compatibility:03}
(\bm{g}, \bm{\eta}) +\langle\bm{\chi}\times\bn, \bm{\eta}\times\bn\rangle =0,\qquad \forall \;\bm{\eta} \in \mathbb{H}_{n,0}(\Omega).
\end{equation}

\medskip

\begin{theorem}
The solution $(\bu,s) \in [L^2(\Omega)]^3\times {\bar{H}^1(\Omega)}$ for the primal problem \eqref{EQ:divcurl-tbv-weakform} is unique.
\end{theorem}

\begin{proof}
It suffices to show that solutions to the homogeneous problem must be trivial. To this end, let $(\bu,s) \in [L^2(\Omega)]^3\times {\bar{H}^1(\Omega)}$ be a solution of \eqref{EQ:divcurl-tbv-weakform} with homogeneous data, i.e.,
\begin{eqnarray}
(\bu, \varepsilon \nabla \varphi + \nabla \times \bpsi) + (\bpsi, \nabla s) = 0 \label{EQ:divcurl-tbv-weakform-homo}
\end{eqnarray}
for all  $\varphi \in H^1_{0c}(\Omega)$ and $\bm{\psi} \in H(curl;\Omega).$ From the Helmholtz decomposition
\eqref{EQ:helmholtz-2:02}, we may choose $\varphi$ and $\bm{\psi}\perp \mathbb{H}_{n,0}(\Omega)$ such that
\begin{equation*}
\begin{split}
\bu =  \nabla \varphi + \varepsilon^{-1} \nabla \times \bpsi,\quad  \mbox{in } \Omega,\\
\nabla\cdot\bpsi = 0, \quad \mbox{in } \Omega,\\
\bpsi\cdot\bn = 0.\qquad \mbox{on } \partial\Omega.
\end{split}
\end{equation*}
Substituting the above into \eqref{EQ:divcurl-tbv-weakform-homo} yields $(\varepsilon\bu, \bu)=0$ so that $\bu\equiv 0$. It follows that $(\bpsi,\nabla s)=0$ for all $\bm{\psi} \in H(curl;\Omega)$ so that $\nabla s =0$, and hence $s=const=0$.
\end{proof}

\medskip
\begin{theorem}
Assume the compatibility conditions \eqref{CompatibilityC:01}, \eqref{CompatibilityC:02}, and \eqref{EQ:divcurl-compatibility:01} hold true.  For any weak solution $(\bu,s) \in [L^2(\Omega)]^3\times {\bar{H}^1(\Omega)}$ of the variational problem \eqref{EQ:divcurl-tbv-weakform}, there hold\textcolor{blue}{s} $s=0$ and that $\bu$ satisfies the div-curl system \eqref{EQ:div-curl-1}-\eqref{EQ:div-curl-TBC2} in the strong form.
\end{theorem}

\begin{proof}
By letting $\bm{\psi}=0$ in \eqref{EQ:divcurl-tbv-weakform} we have
$$
(\bm{u}, \varepsilon \nabla \varphi) = -(f,\varphi) + \sum_{i=1}^L \alpha_i \varphi|_{\Gamma_i}\qquad \forall \varphi\in H^1_{0c}(\Omega),
$$
which, with the integration by parts, leads to 
\begin{eqnarray*}
\nabla\cdot (\varepsilon\bu) &=& f\qquad \mbox{in } \Omega, \\
\langle \varepsilon\bu\cdot\bn_i, 1\rangle_{\Gamma_i} &=& \alpha_i,\;\quad i=1,2,\cdots, L,
\end{eqnarray*}
so that \eqref{EQ:div-curl-1} and \eqref{EQ:div-curl-TBC2} are satisfied. Next, by choosing $\varphi=0$ and $\bm{\psi}=\nabla s$ in \eqref{EQ:divcurl-tbv-weakform} we obtain
$$
(\nabla s, \nabla s) = (\bm{g}, \nabla s) + \langle\bm{\chi}\times\bn, \nabla s \times\bn\rangle,
$$
which, together with the compatibility conditions \eqref{CompatibilityC:01} and \eqref{EQ:divcurl-compatibility:01}, leads to $\nabla s=0$ and hence $s\equiv 0$.

Now, by letting $\varphi=0$ in \eqref{EQ:divcurl-tbv-weakform}, we have  for any $\bm{\psi}\in H(curl;\Omega)$
\begin{eqnarray*}
(\bm{u}, \nabla \times \bm{\psi})  = (\bm{g},\bm{\psi})+\langle\bm{\chi}\times\bn,\bm{\psi}\times\bn\rangle,
\end{eqnarray*}
which leads to
\begin{eqnarray*}
(\nabla\times\bm{u}, \bm{\psi}) + \langle\bu\times\bn, \bm{\psi}\rangle = (\bm{g},\bm{\psi})+\langle\bm{\chi}\times\bn,\bm{\psi}\times\bn\rangle,
\end{eqnarray*}
and hence
\begin{eqnarray}\label{EQ:11:04:222}
\nabla\times\bm{u}  &=& \bm{g},\qquad \mbox{ in } \Omega,\\
\bu\times \bn = \bn\times(\bm{\chi}\times\bn)&=&\bm{\chi},\qquad \mbox{ on } \Gamma.\label{EQ:11:04:223}
\end{eqnarray}
Equation \eqref{EQ:11:04:223} verifies the tangential boundary condition
\eqref{EQ:div-curl-TBC1}, and the curl equation  \eqref{EQ:div-curl-2} is seen from \eqref{EQ:11:04:222}. This completes the proof of the theorem.
\end{proof}

\medskip

\begin{theorem}\label{THM:dual-problem-TBVP} The homogeneous dual problem \eqref{EQ:div-curl-tbvp:dual-problem} has only ``trivial'' solutions for $\lambda$; i.e., if $\lambda\in H_{0c}^1(\Omega)$ and $\bq\in H(curl;\Omega)$ satisfy the weak form \eqref{EQ:div-curl-tbvp:dual-problem}, then we must have $\lambda=0 $ and that $\bq\in\mathbb{H}_{n,0}(\Omega)$ is a harmonic field.
\end{theorem}

\begin{proof}
Let $\lambda\in H_{0c}^1(\Omega)$ and $\bq\in H(curl;\Omega)$ be the solution of the homogeneous dual problem \eqref{EQ:div-curl-tbvp:dual-problem}. Then,
$$
(\bv, \varepsilon \nabla \lambda + \nabla \times \bq) + (\bq, \nabla r)=0,\quad \forall (\bv,r)\in [L^2(\Omega)]^3\times \bar{H}^1(\Omega).
$$
Observe that the test against $r\in \bar{H}^1(\Omega)$ ensures $\nabla\cdot \bq=0$ and $\bq\cdot \bn=0$ on $\Gamma$. It follows that
$$
\varepsilon \nabla \lambda + \nabla \times \bq = 0,\ \nabla\cdot \bq=0, \ \bq\cdot \bn=0 \ \mbox{on } \Gamma.
$$
Now testing the first equation against $\nabla\lambda$ gives
$$
(\varepsilon \nabla \lambda,\nabla\lambda) + (\nabla \times \bq,\nabla\lambda) =0,
$$
which, together with the fact that $(\nabla \times \bq,\nabla\lambda) =0$, leads to $\lambda\equiv 0$. Consequently, the vector field $\bq$ satisfies
$$
\nabla \times \bq = 0,\ \nabla\cdot \bq=0,\ \bq\cdot \bn=0 \ \mbox{on } \partial\Omega.
$$
In other words, $\bq\in\mathbb{H}_{n,0}(\Omega)$ is a harmonic field. This completes the proof of the theorem.
\end{proof}
\medskip

It is known that the dimension of the harmonic space $\mathbb{H}_{\varepsilon n,0}(\Omega)$ is the first Betti number of the domain $\Omega$. The first Betti number is the rank of the first homology group of $\Omega$. It is the number of elements of a maximal set of homologically independent non-bounding cycles in the domain. It is also the dimension of the first de Rham cohomology group of $\Omega$. The dimension of $\mathbb{H}_{\varepsilon n,0}(\Omega)$ is clearly zero if the domain $\Omega$ is simply connected.

\section{Discrete Weak Differential Operators}
The variational problems \eqref{EQ:divcurl-tbv-weakform} and \eqref{EQ:div-curl-tbvp:dual-problem} are formulated with two \textcolor{blue}{principal} differential operators: gradient and curl. This section shall introduce the notion of weak differential operators. These weak differential operators shall be discretized by using piecewise polynomials which lead to discretization schemes for the variational problems.

Let $T$ be a polyhedral domain with boundary $\partial T$. Denote by $\bm{n}$ the unit outward normal direction on $\partial T$. The space of weak functions in $T$ is defined as
\begin{equation*}
W(T) =\{v = \{v_0, v_b\} : v_0 \in L_2(T), v_b \in L_2(\partial T)\},
\end{equation*}
where $v_0$ represents the value of $v$ in the interior of $T$, and $v_b$ represents certain information of $v$ on the boundary $\partial T$.
Similarly, we define $V(T)$ the space of vector-valued weak functions in $T$ given by:
$$
V(T) =\{\bm{v} = \{\bm{v}_0, \bm{v}_b\} : \bm{v}_0 \in [L_2(T)]^3, \bm{v}_b \in [L_2(\partial T)]^3\}.
$$

\subsection{Weak gradient}
The weak gradient of $v \in W(T)$, denoted by $\nabla_w v$, is defined as a continuous linear functional in the Sobolev space $[H^1(T)]^3$ with the following action
\begin{equation*}
\functionalaction{\nabla_w v}{\bm{\varphi}}_T = -(v_0,\nabla\cdot \bm{\varphi})_T + \langle v_b, \bm{\varphi} \cdot \bn \rangle_{\partial T}, \qquad \forall \; \bm{\varphi} \in [H^1(T)]^3.
\end{equation*}
Denote by $P_r(T)$ the space of polynomials on $T$ with total degree $r$ and less.
The discrete weak gradient operator, denoted by $\nabla_{w,r,T} v$, is defined as the unique vector-valued polynomial in $[P_r(T)]^3$ satisfying
\begin{equation}\label{EQ:dis_WeakGradient}
(\nabla_{w,r,T} v, \bm{\varphi})_T = -(v_0,\nabla \cdot \bm{\varphi})_T + \langle v_b , \bm{\varphi}\cdot \bn \rangle_{\partial T},\quad \forall \; \bm{\varphi} \in [P_r(T)]^{3}.
\end{equation}
For smooth $v_0\in H^1(T)$, we have from the usual integration by parts that
\begin{equation*}
(\nabla_{w,r,T}  v, \bm{\varphi})_T = (\nabla v_0, \bm{\varphi})_T + \langle v_b- v_0, \bm{\varphi}\cdot \bn \rangle_{\partial T},\quad \forall \; \bm{\varphi} \in [P_r(T)]^{3}.
\end{equation*}

\subsection{Weak curl}
The weak curl of $\bv\in V(T)$ (see \cite{WangWang_divcurl_2016}), denoted by $\nabla_{w} \times \bv$, is
defined as a bounded linear functional in the Sobolev space $[H^1(T)]^3$ with actions given by
\begin{equation*}
\functionalaction{\nabla_w \times \bv}{\bm{\varphi}}_T = (\bv_0,\nabla \times \bvarphi)_T - \langle \bv_b \times \bn, \bvarphi\rangle_{\partial T}, \quad \forall \; \bm{\varphi} \in [H^1(T)]^3.
\end{equation*}
The discrete weak curl of $\bv \in V (T)$, denoted by $\nabla_{w,r,T}\times \bv$, is defined as the unique vector-valued polynomial in $[P_r(T)]^3$, such that
\begin{equation}\label{EQ:dis_WeakCurl}
(\nabla_{w,r,T} \times \bv, \bvarphi)_T = (\bv_0,\nabla \times \bvarphi)_T - \langle\bv_b \times \bn, \bvarphi\rangle_{\partial T}, \quad \forall \; \bm{\varphi} \in [P_r(T)]^3.
\end{equation}
For sufficiently smooth $\bv_0$ such that $\nabla \times \bv_0 \in [L_2(T)]^3$, we have from the integration by parts that
\begin{equation*}
 (\nabla_{w,r,T} \times \bv, \bvarphi)_T = (\nabla \times \bv_0, \bvarphi)_T - \langle(\bv_b - \bv_0) \times \bn, \bvarphi\rangle_{\partial T}, \quad \forall \;\bm{\varphi} \in [P_r(T)]^3.
\end{equation*}

\section{A Primal-Dual Weak Galerkin Method}
Assume that the domain $\Omega$ is of polyhedral type, and $\T_h=\{T\}$ is a finite element partition of $\Omega$ that is shape regular as described in \cite{WangWang_2016,WangYe_2013}. Denote by $h_T = \diam(T)$ the diameter of the element $T$, and $h = \max_T h_T$ the meshsize of the partition $\T_h = \{T\}$. Denote by $\E_h$ the set of all faces in $\T_h$ so that each $ \sigma\in \E_h$ is either on the boundary of $\Omega$ or shared by two elements $T_1$ and $T_2$. Denote by $\E^0_h = \E_h \textcolor{blue}{\backslash} \partial \Omega$ the set of all interior faces in $\E_h$.
Let $k \geq 0$ be a given integer. On each $T \in \T_h$, we introduce two local weak finite element spaces as follows:
\begin{eqnarray*}
&W(k,T)   &= \{  v = \{\bv_0, \bv_b\} :   v_0 \in  P_k(T),     v_b|_\sigma \in  P_k(\sigma),    \sigma\in (\partial T \cap \E_h)\},\\
&\bV(k,T) &= \{\bv = \{\bv_0, \bv_b\} : \bv_0 \in [P_k(T)]^3, \bv_b|_\sigma \in [P_k(\sigma)]^3\times\bn_\sigma, \sigma\in (\partial T \cap \E_h)\},
\end{eqnarray*}
where $\bn_\sigma$ is a unit normal vector to the face $\sigma$. Note that $\bv_b|_\sigma$ is effectively a vector-valued polynomial of degree $k$ in the tangent space of $\sigma$.
The global weak finite element space is constructed by patching all the local elements $W(k,T)$ ( or  $\bV(k,T)$) through a common value on the interior faces:
\begin{eqnarray*}
&W^k_h &= \{  v = \{v_0, v_b\} : v|_T \in \textcolor{blue}{W}(k, T ), v_b|_{\partial T_1\cap \sigma} = v_b|_{\partial T_2 \cap \sigma}, T \in \T_h, \sigma \in \E^0_h\},\\
&\bV^k_h &= \{\bv = \{\bv_0, \bv_b\} : \bv|_T \in \bV(k, T ), \bv_b|_{\partial T_1\cap \sigma} = \bv_b|_{\partial T_2 \cap \sigma}, T \in \T_h, \sigma \in \E^0_h\},
\end{eqnarray*}
where $\bv_b|_{\partial T_i \cap \sigma}$ is the value of $\bv_b$ on the face $\sigma$ as seen from the element $T_i$. A third finite element space consists of \textcolor{blue}{piecewise} vector-valued polynomials of degree $k$:
\begin{eqnarray*}
&\bm{U}^k_h &= \{\bm{u} : \bm{u} \in [L_2(\Omega)]^3, \bm{u}|_T \in [P_{k}(T)]^3, T\in \T_h\}.
\end{eqnarray*}

Next, we introduce the following finite element spaces:
\begin{equation}\label{EQ:0324:001}
\begin{split}
\bm{U}_h  &=\bm{U}^k_h,\\
M_h       &=\{s=\{s_0,s_b\} \in W^k_h:\  \ (s_0,1)=0\},\\
S_h       &=\{\lambda=\{\lambda_0,\lambda_b\} \in W^k_h:\   \lambda_b|_{\Gamma_0}=0, \ \lambda_b|_{\Gamma_i}=const, \ i=1,\ldots, L \},\\
\bV_h     &=\bV^k_h.
\end{split}
\end{equation}

For functions in $S_h$ and $M_h$, the discrete weak gradient is defined by using \eqref{EQ:dis_WeakGradient} with $r=k$ on each element $T$. Likewise, the discrete weak curl is defined for functions in $\bV_h$ by using \eqref{EQ:dis_WeakCurl} with $r=k$; i.e.,
\begin{eqnarray*}
&&(\nabla_{w,k}  v)|_T =\nabla_{w,k,T} (v|_T ),  \quad v \in W^k_h,\\
&&(\nabla_{w,k} \times \bv)|_T =\nabla_{w,k,T} \times (\bv|_T ), \quad \bv \in \bV^k_h.
\end{eqnarray*}
Note that the weak gradient and the weak curl operators $\nabla_{w,k}$ and $\nabla_{w,k}\times$ are defined by using the same degree of polynomials as the function themselves on each element, as opposed to using polynomials of lower degree in \cite{WangWang_2016}. For simplicity of notation, we shall drop the subscript $k$ from the notations $\nabla_{w,k}$ and $\nabla_{w,k}\times$ in the rest of the paper.

Introduce an approximate bilinear form as follows:
$$
B_h(\bv, r; \varphi,\bpsi): = (\bv, \varepsilon \nabla_w \varphi + \nabla_w \times \bpsi) + (\bpsi_0, \nabla_w r)
$$
for $\bv\in \bm{U}_h,\ r\in M_h,\ \varphi\in S_h,\ \bpsi\in \bV_h$.

\medskip

\begin{algorithm}[PDWG for the div-curl system with tangential BV]
Find $\bu_h, s_h  \in \bm{U}_h\times M_h$ and $\lambda_h, \bq_h \in S_h\times \bV_h$ such that
\begin{equation}\label{EQ:PDWG-3d:01:tangential}
\left\{
\begin{array}{rl}
\s_1(\lambda_h, \bq_h;\varphi,\bpsi) + B_h(\bu_h, s_h; \varphi,\bpsi)&= G(\varphi,\bpsi),\\
-\s_2(s_h,r)+B_h(\bv, r; \lambda_h, \bq_h) & = 0,
\end{array}
\right.
\end{equation}
for all $(\bv, r) \in \bm{U}_h\times M_h$ and
$\varphi, \bpsi \in S_h\times \bV_h$, where

\begin{eqnarray}
\s_1(\lambda_h, \bq_h;\varphi, \bm{\psi})&=&\rho_1\sum_{T}h_T^{-1} \langle \lambda_0 -\lambda_b, \varphi_0 - \varphi_b\rangle_{\partial T} \\
                                 &  + &\rho_2\sum_{T}h_T^{-1} \langle \bm{q}_0 \times \bm{n} -\bm{q}_b \times \bm{n}, \bm{\psi}_0 \times \bm{n} -\bm{\psi}_b \times \bm{n}\rangle_{\partial T}, \nonumber\\
\s_2(s_h,r) &=& \rho_3\sum_{T}h_T^{-1} \langle s_0-s_b,r_0-r_b  \rangle_{\partial T},\\
G(\varphi,\bpsi)&=& (\bm{g},\bpsi_0)+\langle \bm{\chi},\bpsi_b\rangle-(f,\varphi_0) +\sum_{i=1}^L\alpha_i \varphi|_{\Gamma_i}.\label{EQ:0324:300}
\end{eqnarray}
Here $\rho_i>0$ are parameters with prescribed values at user's discretion. The default value for these parameters is $\rho_i=1$.
\end{algorithm}

\section{Element Stiffness Matrix and Load Vector}
In this section, we shall present a formula for the computation of the element stiffness matrix and the element load vector on general 3D polyhedral elements as illustrated in Fig. \ref{Elas_Poly} for the PDWG finite element scheme \eqref{EQ:PDWG-3d:01:tangential} with the lowest order element (i.e., $k=0$), extensions to higher order elements are straightforward.

Let $T\in \T_h$ be a polyhedral element with $N$ lateral faces; i.e., $\partial T = \bigcup_{i=1}^N \sigma_i$. The finite elements for the primal variable $\bu_h$ and the Lagrangian multipliers $ s_h, \lambda_h, \bq_h$ on $T$ are given respectively as follows:
\begin{eqnarray*}
&&\bu_h|_{T} \in [P_0(T)]^3, \\
&&s_h|_{T}=\{s_0, s_b\} \in \{P_0(T)\textcolor{blue}{,\;} P_0(\partial T)\},\\
&&\lambda_h |_{T}= \{\lambda_0 ,\lambda_b\} \in \{P_0(T)\textcolor{blue}{,\;} P_0(\partial T)\}, \\
&&\bq_h|_{T} = \{\bq_0, \bq_b\} \in \{[P_0(T)]^3\textcolor{blue}{,\;} [P_0(\partial T)]^3\}.
\end{eqnarray*}

\begin{wrapfigure}{r}{0.425\textwidth}
\begin{center}
\includegraphics [width=0.425\textwidth]{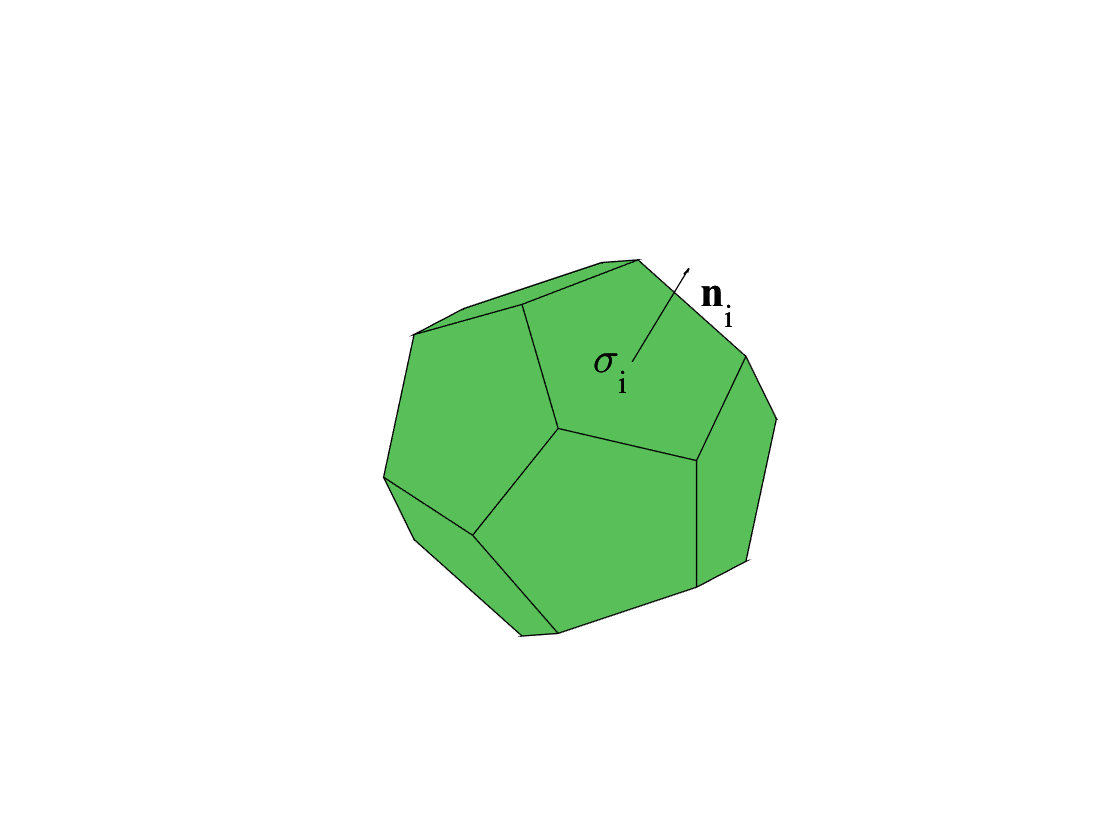}
\end{center}
\caption{3D polyhedron element.}\label{Elas_Poly}
\end{wrapfigure}

We use the following representation for each of them:
\begin{eqnarray*}
\bu_h&=&u_1 \bm{e}^1 + u_2 \bm{e}^2 + u_3 \bm{e}^3, \\
s_0&=&s_0\cdot 1, \ s_b|_{\sigma_i} =s_{b,i},\ i=1,\cdots, N,\\
\lambda_0&=&\lambda_0\cdot 1, \ \lambda_b|_{\sigma_i} =\lambda_{b,i},\ i=1,\cdots, N,\\
\bq_0 &=& q_1\bm{e}^1 + q_2 \bm{e}^2 + q_3 \bm{e}^3,\
\bq_b=\displaystyle\sum_{i=1}^N \sum_{k=1}^2 q_{b,i}^k \bm{e}^k_{b,i}.
\end{eqnarray*}
Note that $\bq_b$ is in fact a vector in the tangent plane of $\sigma_i$, and thus has 2 Dofs.
The $P_0$ vector basis $\bm{e}^k$ on $T$ are as follows:
\begin{equation}
\bm{e}^1
=
\begin{bmatrix}
1 \\
0 \\
0 \\
\end{bmatrix}_T,
\bm{e}^2
=
\begin{bmatrix}
0 \\
1 \\
0 \\
\end{bmatrix}_T,
\bm{e}^3
=
\begin{bmatrix}
0 \\
0 \\
1 \\
\end{bmatrix}_T.
\end{equation}
The tangential basis $\{\bm{e}_{b,i}^1, \bm{e}_{b,i}^2\}$ on face $\sigma_i$ is computed as follows. First we fix a unit normal direction $\tilde{\bn}_i$ to $\sigma_i$, then take an arbitrary vector $\bm{r}$ that is not normal to $\sigma_i$, and compute
\begin{eqnarray*}
\bm{v}_1 =\bm{r} \times \tilde{\bn}_i,\quad
\bm{v}_2 =(\bm{r} \times \tilde{\bn}_i)\times\tilde{\bn}_i.~~~~~~~~~~~~~~~~~~~~~~~~~~~~~~~
\end{eqnarray*}
The tangential basis $\{\bm{e}_{b,i}^1, \bm{e}_{b,i}^2\}$ is chosen as the nomalizations:
\begin{eqnarray*}
\bm{e}_{b,i}^1 = \frac{\bm{v}_1}{|\bm{v}_1|},\; \; \bm{e}_{b,i}^2 = \frac{\bm{v}_2}{|\bm{v}_2|}.
\end{eqnarray*}

Denote by
\begin{eqnarray*}
&&\{u^k\}_{k=1,2,3},\\
&&s_0, \ \{s_{b,i}\}_{i=1,...,N}, \\
&& \lambda_0, \ \{\lambda_{b,i}\}_{i=1,...,N},\\
&& \{{q}^k\}_{k=1,2,3}, \ \{{q}^k_{b,i}\}_{k=1,2,i=1,...,N}
\end{eqnarray*}
 the degree of freedoms on element $T$ for the corresponding variables. For the numerical scheme  \eqref{EQ:PDWG-3d:01:tangential}, we have the following formula for the element stiffness matrix and the load vector:

\begin{theorem}\label{THM:ESM}
The element stiffness matrix and the element load vector for the PDWG scheme \eqref{EQ:PDWG-3d:01:tangential} are given in a block matrix form as follows:
\begin{eqnarray}\label{EQ:ESM:01}
\left[
\begin{array}{c|c|c|c|c|c|c}
 \cellcolor{red!45}A  &\cellcolor{red!25}B&&&&&\\ \hline
 \cellcolor{red!25}B^T&\cellcolor{red!05}C&&&&&\cellcolor{violet!15}D \\ \hline
 &&\cellcolor{green!45}E&\cellcolor{green!25}F&&G\cellcolor{orange!15}&\\ \hline
 &&\cellcolor{green!25}F^T&\cellcolor{green!05}H&&&I\cellcolor{blue!15} \\ \hline
 &&&&J\cellcolor{yellow!45}&K\cellcolor{yellow!25}&\\ \hline
 &&G^T\cellcolor{orange!15}&&\cellcolor{yellow!25}K^T&L\cellcolor{yellow!05}& \\ \hline
 &D^T\cellcolor{violet!15}&&\cellcolor{blue!15}I^T&&&\\
\end{array}
\right]
\left[
\begin{array}{l}
\lambda_0\\
\lambda_{b,i} \\
{q}^k \\
{q}^k_{b,i}\\
s_0\\
s_{b,i}\\
{u}^k\\
\end{array}
\right]
\cong
\left[
\begin{array}{l}
-\int_{T}f \\
~~\sum_{l=1}^L\alpha_l 1_{\Gamma_l \cap \sigma_i}\\
~~\int_T \bm{g}^j \\
~~\langle \bm{\chi}, \bm{e}_{b,i}^j\rangle_{\sigma_i}\\
~~0\\
~~0\\
~~0\\
\end{array}
\right]
\end{eqnarray}
where the block components in \eqref{EQ:ESM:01} are given explicitly as follows when $\rho_i=1$:\\
\begin{eqnarray*}
&&A=\{a\}_{1\times1}, \;a = \displaystyle h_T^{-1} \displaystyle \sum_{i=1}^N |\sigma_i|,\;
B=\{b_{i}\}_{1\times N}, \;b_{i}= -h_T^{-1} |\sigma_i| ,\\
&&C=\{c_{i,j}\}_{N\times N}, \;C =\diag(B),\;
D=\{d_{i,j}\}_{N\times d}, \;d_{i,j}= \bm{e}^j\cdot(\varepsilon \bn_i)|\sigma_i|,\\
&&E=\{x_{k,j}\}_{d\times d}, \;x_{k,j}=\displaystyle h_T^{-1}\sum_{i=1}^N (\bm{e}^k\times \bn_i)\cdot (\bm{e}^j\times \bn_i)|\sigma_i|,\\
&&F=\{f_{i^k, j}\}_{2N\times d}, \;f_{i^k,j} =-h_T^{-1}\bm{e}^k_{bn,i}\cdot (\bm{e}^j\times \bn_i)|\sigma_i|,\;
  G=\{g_{j,i}\}_{d\times N}, \;g_{j,i} = \bm{e}^j\cdot\bn_i |\sigma_i|,\\
&&I =\{x_{i^k,j}\}_{2N \times d},\; x_{i^k,j} = \bm{e}^j\cdot\bm{e}_{bn,i}^k|\sigma_i|,\;
J =-A,\;
K = -B,\;
L =-C, \\
&&H= \begin{bmatrix}
     H_{11}&H_{12}\\
     H_{21}&H_{22}
     \end{bmatrix}_{2N \times 2N},
\end{eqnarray*}
where $H_{kj}, \ {k,j=1,2,}$ are diagonal matrices of size $N\times N$ given as follows:
\begin{eqnarray*}
&&     H_{kj}=\diag (V^{kj}),\; V^{kj}=\{x^{kj}_i\}_{1\times N},\; x^{kj}_i=h_T^{-1}\bm{e}^k_{bn,i}\cdot \bm{e}^j_{bn,i}|\sigma_i|.
\end{eqnarray*}
Here $\bm{e}_{bn,i}^k =\bm{e}^k_{b,i}\times \bn_i$, ${k=1,2}$, $\bn_i$ is the unit outward normal vector of $\sigma_i$.
$d=3$ is the space dimension. Please note the difference between $\bn_i$ and $\tilde{\bn}_i$ in their directions; $\bn_i$ is outward normal to $\sigma_i$ and $\tilde{\bn}_i$ is a prescribed orientation of $\sigma_i$.
\end{theorem}

\begin{proof}
From the definition of the weak gradient \eqref{EQ:dis_WeakGradient}, we have
\begin{eqnarray*}
&&(\nabla_w\{1,0\}, \bm{\phi}) =-(1,{\nabla \cdot \bm{\phi}}) +\langle 0  ,\bm{\phi}\cdot \bn\rangle_{\partial T},\\
&&(\nabla_w\{0,1_{b,i}\}, \bm{\phi})=  \langle 1_{b,i} ,\bm{\phi} \cdot \bn \rangle_{\partial T}, \\
&&(\nabla_w s_h, \bm{\phi})  = -(s_0,{\nabla \cdot \bm{\phi}}) + \langle s_b, \bm{\phi}\cdot\bn\rangle_{\partial T}, \;\forall\; \bm{\phi} \in [P_0]^d.
\end{eqnarray*}
Since $\bm{\phi}\in  [P_0]^d$, then we have $\nabla \cdot \bm{\phi}=0$ and
\begin{eqnarray*}
&&\nabla_w\{1,0\} = \bm{0},\\
&&\nabla_w\{0,1_{b,i}\} =|\sigma_i|\bn_i /|T|,\\
&&\nabla_w s_h = \sum_{i=1}^N s_{b,i}|\sigma_i|\bn_i /|T|,\;\nabla_w \lambda_h = \sum_{i=1}^N \lambda_{b,i}|\sigma_i|\bn_i /|T|.
\end{eqnarray*}
Similarly, by using the definition of the weak curl \eqref{EQ:dis_WeakCurl}, we have
\begin{eqnarray*}
&&\nabla_w  \times \{\bm{e}^k,\bm{0}\}= \bm{0},\\
&&\nabla_w  \times \{\bm{0},\bm{e}^k_{b,i} \}= - \bm{e}^k_{b,i} \times \bn_i |\sigma_i|/|T| = -\bm{e}^k_{bn,i} |\sigma_i|/|T| ,\\
&&\nabla_w  \times \bm{q}_h = -\sum_{k=1}^2\sum_{i=1}^N\bq_{b,i}^k \bm{e}^k_{bn,i} |\sigma_i|/|T|.
\end{eqnarray*}

To derive a formula for the element stiffness matrix and the load vector, we may consider \eqref{EQ:PDWG-3d:01:tangential} with a finite element partition consisting of only one element $T$. By testing this equation with test functions $\bv =\bm{e}^j$, $r =\{1,0\},\;\{0, 1_{b,i}\}$, $\varphi =\{1,0\}$, $\{0, 1_{b,i}\}$, $\bpsi =\{\bm{e}^j,0\}$, $\{0, \bm{e}^j_{b,i}\}$, we easily arrive at the following discrete equations:
\begin{equation}\label{EQ:PDWG-3d:dis}
\left\{
\begin{array}{rl}
&\displaystyle h_T^{-1}\sum_{i=1}^N (\lambda_0 - \lambda_{b,i})|\sigma_i| =-(f,1),\\
&\displaystyle h_T^{-1}(\lambda_0 - \lambda_{b,i})|\sigma_i| + (\sum_{k=1}^d u_k \bm{e}^k,\varepsilon\bm{n}|\sigma_i|/|T|) = \sum_{l=1}^L\alpha_l 1_{\Gamma_l \cap \sigma_i},\\
&\displaystyle h_T^{-1} \sum_{i=1}^N \sum_{k=1}^d  q_k (\bm{e}^k \times \bm{n}_i) \cdot (\bm{e}^j\times \bm{n}_i) |\sigma_i| - h_T^{-1}\sum_{i=1}^N \sum_{k=1}^2 q^k_{b,i}\bm{e}_{bn,i}^k\cdot(\bm{e}^j\times \bm{n}_i)|\sigma_i|\\
&~~~~~~~~~~~~~~~~~~~~~~~~~~~~~~~~~~~~~~~~~~~~~~~~~~~~~~~~~~~~+
\displaystyle\sum_{i=1}^N s_{b,i}|\sigma_i| \bm{e}^j\cdot \bn_i
=
\int_T \bm{g}^j,\\
&\displaystyle -h_T^{-1} \sum_{k=1}^d  q_k (\bm{e}^k \times \bm{n}_i)\cdot \bm{e}^j_{bn,i}|\sigma_i| + h_T^{-1} \sum_{k=1}^2 q^k_{b,i}\bm{e}_{bn,i}^k\cdot\bm{e}_{bn,i}^j|\sigma_i|\\
&~~~~~~~~~~~~~~~~~~~~~~~~~~~~~~~~
-\displaystyle\sum_{k=1}^d u_k \bm{e}^k \cdot \bm{e}^j_{bn,i} |\sigma_i|
=\langle \bm{\chi}, \bm{e}_{b,i}^j\rangle_{\sigma_i},\\
&\displaystyle -h_T^{-1} \sum_{i=1}^N(s_0 - s_{b,i})|\sigma_i|=0,\\
&\displaystyle h_T^{-1}(s_0-s_{b,i})|\sigma_i| +\sum_{k=1}^d q^k \bm{e}^k \cdot \bn_i |\sigma_i|=0,\\
&\displaystyle \sum_{i=1}^N\lambda_{b,i} \bm{e}^j \cdot \varepsilon \bn_i |\sigma_i|
-\sum_{k=1}^2\sum_{i=1}^N q_{b,i}^k \bm{e}^j \cdot \bm{e}^k_{bn,i}|\sigma_i| =0.
\end{array}
\right.
\end{equation}
A matrix version for the above discrete equations gives rise to the formula \eqref{EQ:ESM:01}.
\end{proof}

\begin{remark}
It is not hard to see that the element stiffness matrix is of size $2+2N+2*d+N*(d-1)$.
Therefore, a cubic element would have $32$ dofs in total and a tetrahedral element has $24$ dofs.
In general, for a finite element partition of $N_T$ elements with $N_\sigma$ faces for each element,
the corresponding linear system has dofs no more than $2N_T+2N_\sigma+2N_T*d + N_\sigma*(d-1)$.
While the scheme \eqref{EQ:PDWG-3d:01:tangential} appears to have a lot dofs with piecewise constant approximations, the element stiffness matrix is in fact quite easy to compute. This numerical scheme can be further simplified through condensation or hybridization techniques for fast and parallel computing, which will be addressed in forthcoming papers.
\end{remark}

\section{Solution Existence and Uniqueness}
In this section we show that the PDWG scheme \eqref{EQ:PDWG-3d:01:tangential} has solutions, and the solution is unique for the component $\bu_h$.
\textcolor{blue}{Denote by $Q_0$ the $L^2$ projection operator onto $P_k(T)$. On each face $\sigma \in \partial T$, we use $Q_b$ to denote the $L^2$ projection operator onto $P_k(\sigma)$. Denote by $Q_h$ the projection operator onto the weak finite element space $W(k,T)$ such that
$$
(Q_h w)|_T = \{Q_0 w|_T,Q_b w|_{\pT}\}.
$$
Analogously, we use $\bbQ_0$, $\bbQ_b$ and  $\bbQ_h$ to denote the $L^2$ projection operators onto the vector-valued finite element spaces $[P_k(T)]^3$, $[P_k(\sigma)]^3$, and $\bV(k,T)$, respectively.}

\begin{theorem}\label{THM:ExistenceUniqueness-TBVP-PDWG}
For the finite element spaces $\bm{U}_h, M_h, S_h, \bV_h$ constructed in \eqref{EQ:0324:001},  the solution $(u_h,s_h,\lambda_h,\bm{q}_h)$ of the primal-dual weak Galerkin finite element scheme \eqref{EQ:PDWG-3d:01:tangential} is unique for all the components except $\bm{q}_h$.  The solution for $\bm{q}_h$ is unique up to a continuous piecewise $[P_k]^3$ harmonic field in $\HnHarmonic$.
\end{theorem}

\begin{proof}
For any solution $(\bm{u}_h,s_h, \lambda_h,\bq_h)$ of \eqref{EQ:PDWG-3d:01:tangential} arising from the finite element spaces $\bm{U}_h, M_h, S_h, \bV_h$ with homogeneous data, the following clearly holds true:
\begin{eqnarray}
&& \s_1(\lambda_h, \bm{q}_h;\lambda_h, \bm{q}_h) = 0, \ \s_2(s_h,s_h)=0, \label{eq:11:03:100}\\
&& (\bm{u}_h, \varepsilon\nabla_w \varphi + \nabla_w \times\bm{\psi}) +(\bm{\psi}_0,\nabla_w s_h) = 0,\; \forall\;\varphi, \bm{\psi} \in  S_h\times \bV_h, \label{eq:11:03:101}\\
&&(\bm{q}_0, \nabla_w r)+(\bm{v}, \varepsilon \nabla_w \lambda_h +\nabla_w \times \bm{q}_h) = 0,\; \forall \; (\bm{v}, r)\in \bm{U}_h\times M_h.\label{eq:11:03:102}
\end{eqnarray}
From \eqref{eq:11:03:100}, we have
\begin{eqnarray} \lambda_0=\lambda_b,\;\bm{q}_0\times \bm{n}=\bm{q}_b\times \bm{n},\; s_0=s_b  \text{ on } \partial T, \ \forall T\in\T_h
\end{eqnarray}
so that $\lambda_0\in H^1_{0c}(\Omega)$, $\bm{q}_0\in H(curl;\Omega)$, $s_0\in H^1(\Omega)$ and hence
\begin{eqnarray}
\nabla \lambda_0 = \nabla_w \lambda_h,\; \nabla \times \bm{q}_0 = \nabla_w \times \bm{q}_h, \; \nabla s_0 = \nabla_w s_h.
\end{eqnarray}
By letting $r=0$ in \eqref{eq:11:03:102} we obtain
\begin{eqnarray}\label{EQ:0324:200}
\varepsilon \nabla \lambda_0 + \nabla\times \bm{q}_0 =0\quad \mbox{in } \Omega.
\end{eqnarray}
It follows from $\lambda_0\in H^1_{0c}(\Omega)$ that
\begin{eqnarray*}
(\varepsilon \nabla \lambda_0 + \nabla\times \bm{q}_0, \nabla \lambda_0 )
&=&(\varepsilon \nabla \lambda_0, \nabla \lambda_0 ) + (\nabla\times \bm{q}_0, \nabla \lambda_0 )\\
&=&(\varepsilon \nabla \lambda_0, \nabla \lambda_0 ) + \langle \bm{q}_0, \nabla\lambda_0\times\bn\rangle\\
& = & (\varepsilon \nabla \lambda_0, \nabla \lambda_0 ).
\end{eqnarray*}
Substituting \eqref{EQ:0324:200} into the above identity yields
\begin{eqnarray}
(\varepsilon \nabla \lambda_0, \nabla \lambda_0 )=0,
\end{eqnarray}
which leads to
\begin{eqnarray}
\nabla \lambda_0 = \bm{0},
\end{eqnarray}
so that  $\lambda_0\equiv 0$ as $\lambda_0 \in H^1_{0c}(\Omega)$. This further implies that $\lambda_b\equiv 0$ and $\nabla \times \bm{q}_0 =0$ in $\Omega$.

 Next, from \eqref{eq:11:03:102}, the Lagrangian multiplier $\bm{q}_0$ is seen to satisfy the following equation:
$$
(\bm{q}_0, \nabla_w r)=0,\qquad \forall\; r\in M_h,
$$
which implies $\bm{q}_0\in H(div;\Omega)$, $\nabla\cdot \bm{q}_0 = 0$, and ${\bm q}_0\cdot\bn =0$ on the domain boundary so that $\bm{q}_0 \in \mathbb{H}_{n,0}(\Omega)$.

Finally, from the Helmholtz decomposition \eqref{EQ:helmholtz-2:02} in Theorem \ref{THM:helmholtz-2:02}, there exist $\tilde\varphi\in H_{0c}^1(\Omega)$ and ${\tilde\bpsi}\in H(curl;\Omega)$ such that
$$
\bu_h =   \nabla \tilde \varphi + \varepsilon^{-1} \nabla \times \tilde\bpsi,\ \ \nabla\cdot\tilde\bpsi=0, \ \tilde\bpsi\cdot\bn=0 \ \mbox{on } \partial\Omega.
$$
By letting $\varphi = Q_h \tilde\varphi \in S_h$ and $\bpsi=\textcolor{blue}{\bbQ_h} \tilde\bpsi \in \bV_h$, from the above equation we obtain
$$
\bu_h =   \nabla_w \varphi + \varepsilon^{-1} \nabla_w \times \bpsi.
$$
As $\nabla_w s_h= \nabla s_0$, from the above equation and \eqref{eq:11:03:101} we have
\begin{eqnarray*}
0&=& (\varepsilon\bu_h, \nabla_w \varphi + \varepsilon^{-1} \nabla_w \times \bpsi) + (\bpsi_0, \nabla_w s_h ) \\
&=& (\varepsilon\bu_h, \bu_h) + (Q_0\tilde\bpsi, \nabla s_0)\\
&=& (\varepsilon\bu_h, \bu_h) + (\tilde\bpsi, \nabla s_0)\\
&=& (\varepsilon\bu_h, \bu_h),
\end{eqnarray*}
where we have also used the fact that $\nabla\cdot\tilde\bpsi=0$ and $\tilde\bpsi\cdot\bn=0$ on $\partial\Omega$. It follows that $\bu_h\equiv 0$.

Going back to \eqref{eq:11:03:101}, from $\bu_h=0$ we obtain
$$
(\bpsi_0, \nabla_w s_h) = 0 \quad \forall \bpsi\in \bV_h,
$$
which leads to $\nabla s_0 = \nabla_w s_h =0$ so that $s_0\equiv 0$ and hence $s_b\equiv 0$. This completes the proof of the solution uniqueness for $\bu_h, s_h$, and $\lambda_h$. The solution for the Lagrangian multiplier $\bq_h$ is unique up to a continuous piecewise $[P_k]^3$ polynomial in the harmonic space $\HnHarmonic$.
\end{proof}

The proof of the solution uniqueness indicates that the kernel of the matrix for the PDWG finite element scheme \eqref{EQ:PDWG-3d:01:tangential} consisting of functions in the following form:
$$
(u_h,s_h,\lambda_h,\bm{q}_h) = (0, 0, 0, \boldeta_h) \in \bm{U}_h\times M_h \times S_h \times \bV_h,
$$
where $\boldeta_h \in \HnHarmonic$ is continuous piecewise polynomials in $[P_k]^3$. For simplicity, we denote this kernel space by $\bH_h\subset \HnHarmonic$. For the case of $k=0$ (i.e., piecewise constant approximating functions), the kernel space $\bH_h$ would consist of a constant vector in $\mathbb{R}^3$ satisfying the homogeneous normal boundary condition on $\partial\Omega$. Thus, we have $\bH_h=\{\bm{0}\}$ in nearly all the applications, so that the solution for $\bq_h$ is in fact unique in the usual sense.

\begin{theorem}\label{THM:ExistenceUniqueness-TBVP-PDWG-existence}
The primal-dual weak Galerkin finite element scheme \eqref{EQ:PDWG-3d:01:tangential} has at least one  solution $(u_h,s_h,\lambda_h,\bm{q}_h)$ in the finite element spaces $\bm{U}_h, M_h, S_h, \bV_h$ given in \eqref{EQ:0324:001}.
\end{theorem}

\begin{proof}
The linear system \eqref{EQ:PDWG-3d:01:tangential} has solutions as long as the following compatibility condition is satisfied:
$$
G(\varphi,\bpsi)=0\qquad \forall \bpsi\in \bH_h, \ \varphi=0.
$$
In fact, from \eqref{EQ:0324:300} and the compatibility condition \eqref{EQ:divcurl-compatibility:02}, we have
$$
G(0,\boldeta_h)= (\bm{g},\boldeta_h)+\langle \bm{\chi},\boldeta_h\rangle = 0\qquad \forall \boldeta_h\in \bH_h,
$$
which completes the proof of the theorem.
\end{proof}

\section{Error Equations}
For the numerical approximation $(\bm{u}_h,s_h, \lambda_h,\bq_h) \in \bm{U}_h\times M_h\times S_h\times \bV_h$
 of the div-curl system with tangential boundary condition arising from the PDWG scheme \eqref{EQ:PDWG-3d:01:tangential}, we introduce the following error functions:
$$
e_{\bu}= \bbQ_0\bu - \bu_h, \ e_s=Q_h s - s_h,\  e_\lambda=Q_h\lambda - \lambda_h, \ e_{\bm{q}}=\bbQ_h\bm{q}-\bm{q}_h,
$$
where $(\bm{u},s)$ is the exact solution of the variational problem \eqref{EQ:divcurl-tbv-weakform}-\eqref{EQ:div-curl:tbvp:gform}, and $(\lambda,\bm{q})$ is the exact solution of the dual problem \eqref{EQ:div-curl-tbvp:dual-problem}. Recall that  we have $s=0$,  $\lambda=0$, and shall take $\bm{q}=\bm{0}$ (note that the solution for $\bq$ is non-unique).
It is clear that $(e_{\bu},e_s, e_\lambda, e_{\bm{q}})\in \bm{U}_h\times M_h \times S_h \times \bV_h $.

\medskip
\begin{lemma}
The following equations are satisfied by $(e_{\bu}, e_s, e_\lambda, e_{\bm{q}})\in \bm{U}_h\times M_h \times S_h \times \bV_h $:
\begin{eqnarray}
\s_1(e_\lambda,e_{\bm{q}};\varphi, \bm{\psi})
    + B_h(e_{\bu}, e_s; \varphi,\bpsi) &=& \ell_{\bu}(\varphi, \bpsi)\quad \forall \varphi\in S_h, \bm{\psi} \in \bV_h,
\label{EQ:div-curl:EE:November-03:ee:01}\\
-\s_2(e_s, r) + B_h(\bm{v}, r; e_\lambda, e_{\bm{q}}) &=& 0 \quad \forall  \bm{v} \in \bm{U}_h, r\in M_h, \label{EQ:div-curl:EE:November-03:ee:02}
\end{eqnarray}
where
$$
\ell_{\bu}(\varphi, \bpsi):= \langle e_{\bu}, \varepsilon\bn(\varphi_0-\varphi_b) + (\bpsi_b-\bpsi_0)\times\bn\rangle_{\partial\T_h}.
$$
The equations \eqref{EQ:div-curl:EE:November-03:ee:01} and \eqref{EQ:div-curl:EE:November-03:ee:02} are called error equations. Here
\begin{equation}\label{EQ:bdry-ip}
\langle e_{\bu}, \varepsilon\bn(\varphi_0-\varphi_b) + (\bpsi_b-\bpsi_0)\times\bn\rangle_{\partial\T_h}:=\sum_{T\in\T_h} \langle e_{\bu}, \varepsilon\bn(\varphi_0-\varphi_b) + (\bpsi_b-\bpsi_0)\times\bn\rangle_{\partial T}.
\end{equation}
\end{lemma}

\begin{proof}
We first derive an equation for the $L^2$ projections of the exact solution. To this end, for the exact solution $(\bu,s=0)$,  we have
\begin{equation}\label{EQ:div-curl:EE:November-03:100}
\begin{split}
B_h(\bbQ_0\bu, Q_h s; \varphi,\bpsi) =& (\bbQ_0\bu, \varepsilon \nabla_w \varphi + \nabla_w \times \bpsi) + (\bpsi_0, \nabla_w Q_h s)\\
=&(\bbQ_0\bu, \varepsilon \nabla \varphi_0 + \nabla \times \bpsi_0)\\
 & + \langle \bbQ_0\bu, \varepsilon\bn(\varphi_b-\varphi_0) + (\bpsi_0-\bpsi_b)\times\bn\rangle_{\partial\T_h}\\
=& (\bu, \varepsilon \nabla \varphi_0 + \nabla \times \bpsi_0)\\
& + \langle \bbQ_0\bu, \varepsilon\bn(\varphi_b-\varphi_0) + (\bpsi_0-\bpsi_b)\times\bn\rangle_{\partial\T_h}\\
=& -(\nabla\cdot(\varepsilon\bu),\varphi_0) + (\nabla\times\bu, \bpsi_0)\\
& + \langle \bu, \varepsilon\bn(\varphi_0-\varphi_b) + (\bpsi_b-\bpsi_0)\times\bn\rangle_{\partial\T_h}\\
& + \langle \bbQ_0\bu, \varepsilon\bn(\varphi_b-\varphi_0) + (\bpsi_0-\bpsi_b)\times\bn\rangle_{\partial\T_h}\\
& - \langle \bm{u}, \bm{\psi}_b\times\bn\rangle_{\partial\Omega}+\sum_{i=1}^L\alpha_i \varphi_b|_{\Gamma_i}\\
=& - (f,\varphi_0) +(\bm{g},\bm{\psi}_0)\\
&+ \langle \bu-Q_h\bu, \varepsilon\bn(\varphi_0-\varphi_b) + (\bpsi_b-\bpsi_0)\times\bn\rangle_{\partial\T_h}\\
&+ \langle (\bm{u}\times\bn)\times\bn, \bm{\psi}_b\times\bn\rangle_{\partial\Omega} + \sum_{i=1}^L\alpha_i \varphi_b|_{\Gamma_i}\\
=& \langle \bm{\chi}\times\bn, \bm{\psi}_b\times\bn\rangle_{\partial\Omega} - (f,\varphi_0) +(\bm{g},\bm{\psi}_0)+\sum_{i=1}^L\alpha_i \varphi_b|_{\Gamma_i}\\
&+ \langle \bu-Q_h\bu, \varepsilon\bn(\varphi_0-\varphi_b) + (\bpsi_b-\bpsi_0)\times\bn\rangle_{\partial\T_h},
\end{split}
\end{equation}
where we have used the usual integration by parts and the fact that $\bu$ satisfies the div-curl system \eqref{EQ:div-curl-1}-\eqref{EQ:div-curl-TBC2}, plus $\langle\bu, \varepsilon\bn\varphi_b\rangle_{\T_h} =\sum_i\alpha_i \varphi_b|_{\Gamma_i}$ and $\bu\times\bn = \bm{\chi}$ on $\partial\Omega$.
Thus, from \eqref{EQ:div-curl:EE:November-03:100} and the fact that $\lambda=0$ and $\bm{q}=0$ we arrive at
\begin{equation}\label{EQ:div-curl:EE:November-03:ee:01original}
\begin{split}
&\s_1(Q_h\lambda-\lambda_h, \bbQ_h\bm{q}-\bm{q}_h;\varphi, \bm{\psi})
+ B_h(\bbQ_0\bu-\bu_h, Q_h s - s_h; \varphi,\bpsi) \\
= & \langle \bu-\bbQ_0\bu, \varepsilon\bn(\varphi_0-\varphi_b) + (\bpsi_b-\bpsi_0)\times\bn\rangle_{\partial\T_h},
\;\forall \;\varphi, \bm{\psi}.
\end{split}
\end{equation}

The second error equation can be easily seen as follows:
\begin{equation}\label{EQ:div-curl:EE:November-03:ee:02original}
-\s_2(Q_h s-s_h, r)
+ B_h(\bm{v}, r; Q_h\lambda - \lambda_h, \bbQ_h\bm{q}- \bm{q}_h) = 0, \;\forall \; \bm{v}, r,
\end{equation}
where we have used the fact that $s=0, \ \bm{q}=0$, and $\lambda=0$. The equations
\eqref{EQ:div-curl:EE:November-03:ee:01original}-\eqref{EQ:div-curl:EE:November-03:ee:02original} lead to \eqref{EQ:div-curl:EE:November-03:ee:01}- \eqref{EQ:div-curl:EE:November-03:ee:02}.
\end{proof}

\section{Error Estimates}
In the space $M_h$ and $S_h\times\bV_h$, we introduce the following semi-norms
\begin{eqnarray}
&&\3bar s \3bar^2 = \s_2(s, s),\\
&&\3bar (\lambda,\bq) \3bar^2 =  \s_1(\lambda,\bq;\lambda, \bq).
\end{eqnarray}

\begin{lemma}
The error functions $e_s$ and $(e_\lambda,e_\bq)$ have the following error estimates
\begin{equation}\label{EQ:error-estimate-t-part01:new}
\3bar(e_\lambda, e_{\bm{q}})\3bar+ \3bar e_s\3bar \leq Ch^{k+\theta}\|\bu\|_{k+\theta},
\end{equation}
where $\theta\in (1/2,1]$ and $k$ is the order of polynomials in the finite element space $\bm{U}_h$.
\end{lemma}
\begin{proof}
By choosing $\varphi= e_\lambda$, $\bm{\psi} = e_{\bm{q}}$ in \eqref{EQ:div-curl:EE:November-03:ee:01}, and
\textcolor{blue}{$\bm{v} = e_{\bu}$}, $r = e_s$ in \eqref{EQ:div-curl:EE:November-03:ee:02}, the two resulting equations give rise to
\begin{equation*}
\begin{split}
\s_1(e_\lambda, e_{\bm{q}}; e_\lambda, e_{\bm{q}}) + \s_2(e_s, e_s)
= &\langle \bu-\bbQ_0\bu, \varepsilon\bn(e_{\lambda,0}-e_{\lambda,b}) + (e_{\bm{q},b}-e_{\bm{q},0})\times\bn\rangle_{\partial\T_h},
\end{split}
\end{equation*}
which, from the Cauchy-Schwarz inequality, leads to
\begin{eqnarray}\label{EQ:error-estimate-t-part01}
&\s_1(e_\lambda, e_{\bm{q}}; e_\lambda, e_{\bm{q}}) +\s_2(e_s, e_s) &\leq C \sum_{T\in\T_h} h_T\|\bu-\bbQ_0\bu\|_{\pT}^2\\
&&\leq Ch^{2k+2\theta}\|\bu\|_{k+\theta}^2\nonumber
\end{eqnarray}
which verifies the error estimate \eqref{EQ:error-estimate-t-part01:new}.
\end{proof}
\medskip

Next we derive an estimate for the error function $e_\bu$. To this end, from the Helmholtz decomposition \eqref{EQ:helmholtz-2:02}, there exist two functions $\tilde\phi\in H_{0c}^1(\Omega)$ and $\bm{\tilde\psi}\in  H(curl;\Omega)\cap H(div;\Omega)$ such that
$$
e_\bu = \varepsilon^{-1} \nabla\times\bm{\tilde\psi} + \nabla\tilde\phi,\ \nabla\cdot\tilde\bpsi=0,\ \tilde\bpsi\cdot\bn=0\ \ \mbox{on } \partial\Omega.
$$
Assume that the following $H^\alpha$-regularity holds true for this Helmholtz decomposition:
\begin{equation}\label{EQ:regularity-assumption-helmholtz-01}
\|\bm{\tilde\psi}\|_\alpha + \|\tilde\phi\|_\alpha \leq C \|e_\bu\|_0,
\end{equation}
with some $\alpha\in (1/2, 1]$.

\medskip
\begin{theorem}\label{thm.error.u}
Let $\bm{u} \in [L^{2}(\Omega)]^3 $ be the solution of \eqref{EQ:div-curl-1}-\eqref{EQ:div-curl-TBC2}, and $\bm{u}_h \in \bm{U}_h$ be the solution of the PDWG scheme \eqref{EQ:PDWG-3d:01:tangential}. Then, the following error estimate holds true:
\begin{eqnarray}\label{eq:error_estimates.l2.u}
\| \varepsilon^{\frac12}(\bbQ_0\bm{u} - \bm{u}_h)\| \leq Ch^{k+\theta+\alpha-1} \|\bu\|_{k+\theta},
\end{eqnarray}
where $\alpha\in (1/2, 1]$ is the regularity parameter in \eqref{EQ:regularity-assumption-helmholtz-01}, $k+\theta$ is the regularity of $\bu$ with some $\theta\in (1/2,1]$ and $k$ is the order of polynomials for the finite element space $\bm{U}_h$.
\end{theorem}
\begin{proof}
By choosing $\bm{\psi}=\textcolor{blue}{\bbQ_h}\bm{\tilde\psi}$ and $\varphi= Q_h\tilde\phi$ in equation \eqref{EQ:div-curl:EE:November-03:ee:01}, we obtain
\begin{eqnarray}\label{EQ:div-curl:EE:November-03:ee:01-new}
B_h(e_\bu, e_s; \varphi,\bpsi)
&=& \langle \bu-\bbQ_0\bu, \varepsilon\bn(\varphi_0-\varphi_b) + (\bpsi_b-\bpsi_0)\times\bn\rangle_{\partial\T_h}\\
&&-\s_1(e_\lambda, e_{\bm{q}};\varphi, \bm{\psi}).\nonumber
\end{eqnarray}
On the other hand, we have
\begin{eqnarray}\label{EQ:11:05:300}
B_h(e_\bu, e_s; \varphi,\bm{\psi})&=&(e_\bu, \varepsilon\nabla_w Q_h\tilde\phi + \nabla_w\times Q_h\bm{\tilde\psi})+(Q_0\bm{\tilde\psi},\nabla_w e_s)\\
&= & (e_\bu, \varepsilon\nabla \tilde\phi + \nabla\times \bm{\tilde\psi})+(Q_0\bm{\tilde\psi},\nabla_w e_s) \nonumber\\
&= & (e_\bu, \varepsilon e_\bu) + (Q_0\bm{\tilde\psi},\nabla_w e_s),   \nonumber
\end{eqnarray}
and from the definition of the weak gradient
\begin{equation*}
\begin{split}
(Q_0\bm{\tilde\psi},\nabla_w e_s) = &(Q_0\bm{\tilde\psi},\nabla e_{s,0}) + \langle Q_0\bm{\tilde\psi}\cdot\bn, e_{s,b}-e_{s,0}\rangle_{\partial\T_h} \\
= & (\bm{\tilde\psi},\nabla e_{s,0}) + \langle Q_0\bm{\tilde\psi}\cdot\bn, e_{s,b}-e_{s,0}\rangle_{\partial\T_h} \\
= & -(\nabla\cdot\bm{\tilde\psi},\nabla e_{s,0}) +\langle \bm{\tilde\psi}\cdot\bn, e_{s,0} \rangle_{\partial\T_h} + \langle Q_0\bm{\tilde\psi}\cdot\bn, e_{s,b}-e_{s,0}\rangle_{\partial\T_h}\\
= & \langle \bm{\tilde\psi}\cdot\bn, e_{s,0} - e_{s,b}\rangle_{\partial\T_h} + \langle Q_0\bm{\tilde\psi}\cdot\bn, e_{s,b}-e_{s,0}\rangle_{\partial\T_h}\\
= & \langle (\bm{\tilde\psi} - Q_0\bm{\tilde\psi})\cdot\bn, e_{s,0} - e_{s,b}\rangle_{\partial\T_h}.
\end{split}
\end{equation*}
Substituting the above into \eqref{EQ:11:05:300} then \eqref{EQ:div-curl:EE:November-03:ee:01-new} yields
\begin{equation*}
\begin{split}
\|\varepsilon^{\frac12}e_\bu\|^2 = & B_h(e_\bu, e_s; \varphi,\bm{\psi}) - \langle (\bm{\tilde\psi} - Q_0\bm{\tilde\psi})\cdot\bn, e_{s,0} - e_{s,b}\rangle_{\partial\T_h}\\
= & \langle \bu-\bbQ_0\bu, \varepsilon\bn(\varphi_0-\varphi_b) + (\bpsi_b-\bpsi_0)\times\bn\rangle_{\partial\T_h} - \s_1(e_\lambda, e_{\bm{q}};\varphi, \bm{\psi}) \\
& - \langle (\bm{\tilde\psi} - Q_0\bm{\tilde\psi})\cdot\bn, e_{s,0} - e_{s,b}\rangle_{\partial\T_h},
\end{split}
\end{equation*}
which gives
\begin{equation*}
\begin{split}
\|\varepsilon^{\frac12}e_\bu\|^2 \le &
Ch^{\alpha-1} (\|\bu-\bbQ_0\bu\|_0 + h^\theta \|\bu-\bbQ_0\bu\|_\theta + \3bar (e_\lambda, e_{\bm{q}})\3bar)  (\|\tilde\varphi\|_\alpha + \|\bm{\tilde\psi}\|_\alpha)\\
& + C h^\alpha \3bar e_{s}\3bar \|\bm{\tilde\psi}\|_\alpha,
\end{split}
\end{equation*}
which, by using the regularity assumption \eqref{EQ:regularity-assumption-helmholtz-01}, leads to
\begin{equation*}
\begin{split}
\|\varepsilon^{\frac12}e_\bu\| \le &
Ch^{\alpha-1} (\|\bu-\bbQ_0\bu\|_0 + h^\theta \|\bu-\bbQ_0\bu\|_\theta + \3bar (e_\lambda, e_{\bm{q}})\3bar) \\
& + C h^{\alpha} \3bar e_{s}\3bar.
\end{split}
\end{equation*}
Substituting \eqref{EQ:error-estimate-t-part01:new} into the above estimate gives
\begin{equation*}
\begin{split}
\|\varepsilon^{\frac12}e_\bu\| \le &
Ch^{k+\theta+\alpha-1} \|\bu\|_{k+\theta}.
\end{split}
\end{equation*}
This completes the proof of the theorem.
\end{proof}

\section{Numerical Experiments}
The goal of this section is to numerically demonstrating the performance of the PDWG finite element method \eqref{EQ:PDWG-3d:01:tangential}. Various numerical examples are employed in the numerical experiments; some are defined on convex domains and the others are on non-convex polyhedral domains with various topological properties. In the case of convex domain, we use a test problem defined on the unit cube $\Omega = (0,1)^3$. The non-convex domains include domains with single or multiple holes. The PDWG scheme \eqref{EQ:PDWG-3d:01:tangential} was implemented by using the lowest order element; i.e., $k = 0$, so that the vector field $\bm{u}$ is approximated by piecewise constant functions.

\subsection{Tests on the unit cubic domain}
The computational domain is given by $\Omega=(0,1)^3$, which is partitioned into cubic elements with
different meshsize $h$. The div-curl system with $\varepsilon=I$ was considered.  Our test examples assumed the following exact solutions:
\begin{equation*}
\begin{array}{lll}
&\bm{u}_1 =
\begin{bmatrix}
y(1-y)z(1-z)\\
x(1-x)z(1-z)\\
x(1-x)y(1-y)
\end{bmatrix}
,
&\bm{u}_2 =
\begin{bmatrix}
\sin(\pi x)\sin(\pi y)\sin(\pi z)\\
xyz\\
(x+1)(y+1)(z+1)
\end{bmatrix},\\
&\bm{u}_3 =
\begin{bmatrix}
y(1-y)z(1-z)\\
x(1-x)z(1-z)\\
r^{\frac{2}{3}}\sin(2\theta)(1-x)(1-y)
\end{bmatrix},
&\bm{u}_4 =
\begin{bmatrix}
\nabla \textcolor{blue}{(r^{\frac{2}{3}}\sin(\frac{2}{3}\theta))}
\end{bmatrix},
\end{array}
\end{equation*}
where the cylindrical coordinates are used in the third and fourth test cases; i.e., $r=\sqrt{x^2+y^2}$ and $\theta=\tan^{-1}(y/x)$. Note that the vector field $\bm{u}_3$ is in $H^{1+\frac23-\epsilon}(\Omega)$ and $\bm{u}_4$ is in $H^{\frac23 -\epsilon}(\Omega)$ with $\epsilon >0$.
The test examples with $\bm{u}_1$, $\bm{u}_3$, and $\bm{u}_4$ as exact solutions have been considered in \cite{LiYeZhang_2018,YeZhangZhu_2020}. The right-hand side functions $f$ and $\bm{g}$ are chosen to match the exact solution for each test example. The tangential boundary condition was imposed on the boundary $\Gamma=\partial \Omega$.

The approximation error and convergence rates for the lowest order PDWG scheme \eqref{EQ:PDWG-3d:01:tangential} are reported in Table \ref{table1.1}. A super-convergence of order $2$ was clearly seen for the test case with exact solution $\bm{u}_1$. Observe that the tangential boundary condition is of homogenous for the case of $\bm{u}_1$. For the case of $\bm{u}_2$ and $\bm{u}_3$, the numerical convergence has an order higher than the optimal order of $r =1$, which outperforms the convergence theory developed in previous sections.
For $\bm{u}_4$, the numerical results show that the PDWG scheme performs better than the theorerical rate of convergence $r=\frac23$.

\begin{table}[!h]
\begin{center}
\caption{Error and convergence performance of the PDWG scheme  for the div-curl systems on cubic meshes. $r$ refers to the order of convergence in $O(h^r).$ }\label{table1.1}

\begin{tabular}{||c|cc|cc||}
\hline
\multicolumn{3}{|>{\columncolor{mypink}}c|}{ $\bm{u_1}$ }&\multicolumn{2}{|>{\columncolor{green!15}}c|}{$\bm{u_2}$}\\
\hline
$n$ & $\|e_{\bm{u}} \|_{0}$ & $r=$ & $\|e_{\bm{u}}\|_{0}$ & $r=$   \\
\hline
2  &   2.48e-02 &  -    &    1.57e-01 &     \\
4  &   5.34e-03 &  2.22 &    7.64e-02 & 1.04\\
8  &   1.24e-03 &  2.11 &    2.75e-02 & 1.47\\
16 &   3.03e-04 &  2.03 &    8.25e-03 & 1.74 \\
\hline
\multicolumn{3}{|>{\columncolor{yellow!20}}c|}{ $\bm{u_3}$ }&\multicolumn{2}{|>{\columncolor{blue!20}}c|}{$\bm{u_4}$}\\
\hline
$n$ & $\|e_{\bm{u}} \|_{0}$ & $r=$ & $\|e_{\bm{u}}\|_{0}$ & $r=$ \\
\hline
2  &   2.27e-02 &  -     &   5.54e-02  & -     \\
4  &   6.55e-03 &  1.79  &   4.41e-02  & 0.33  \\
8  &   3.03e-03 &  1.11  &   3.00e-02  & 0.56  \\
16 &   1.38e-03 &  1.13  &   1.66e-02  & 0.85  \\
\hline
\end{tabular}
\end{center}
\end{table}

\subsection{Numerical tests on domains with complex topology}
The test problems involve three different type of domains, including (a)
a toroidal domain given specifically by
$\Omega_A = (-2,2)^3/H_A$, with $H_A=[-1,1]\times[-1,1]\times[-2,2]$; (b)
a cubic domain with a small hole inside, given by $\Omega_B = (-2,2)^3\textcolor{blue}{\backslash}H_B$, with $H_B=[-1,1]^3$; and (c) a domain with two holes given by $\Omega_C =(-2,2)\times (-2,6)\times (0, 1)/(H_C\cup H_D)$, with $H_C=[-1.5,1.5]\times [-1.5, 1.5]\times[0, 1]$ and $H_D =[-1.5,1.5] \times[2.5, 5.5]\times[0, 1]$. The domains are illustrated in Figure \ref{Hole-domain}.

The exact solutions of the test problems are given as follows:
\begin{equation*}
\bm{u}_5 =
\begin{bmatrix}
x + y +z\\
x-z\\
x + 3y
\end{bmatrix},\;\;
\bm{u}_6 =
\begin{bmatrix}
\sin(x)\sin(y)\sin(z)\\
xyz\\
(x+1)(y+1)(z+1)
\end{bmatrix}.
\end{equation*}
The right-hand side functions $f$ and $\bm{g}$ are computed to match the exact solutions for each test case. The tangential boundary condition is imposed on the boundary $\Gamma=\partial \Omega$.

\begin{figure}[h!]
\begin{center}
\subfigure[]{\label{Fig.sub3.l}
\includegraphics [width=0.3\textwidth]{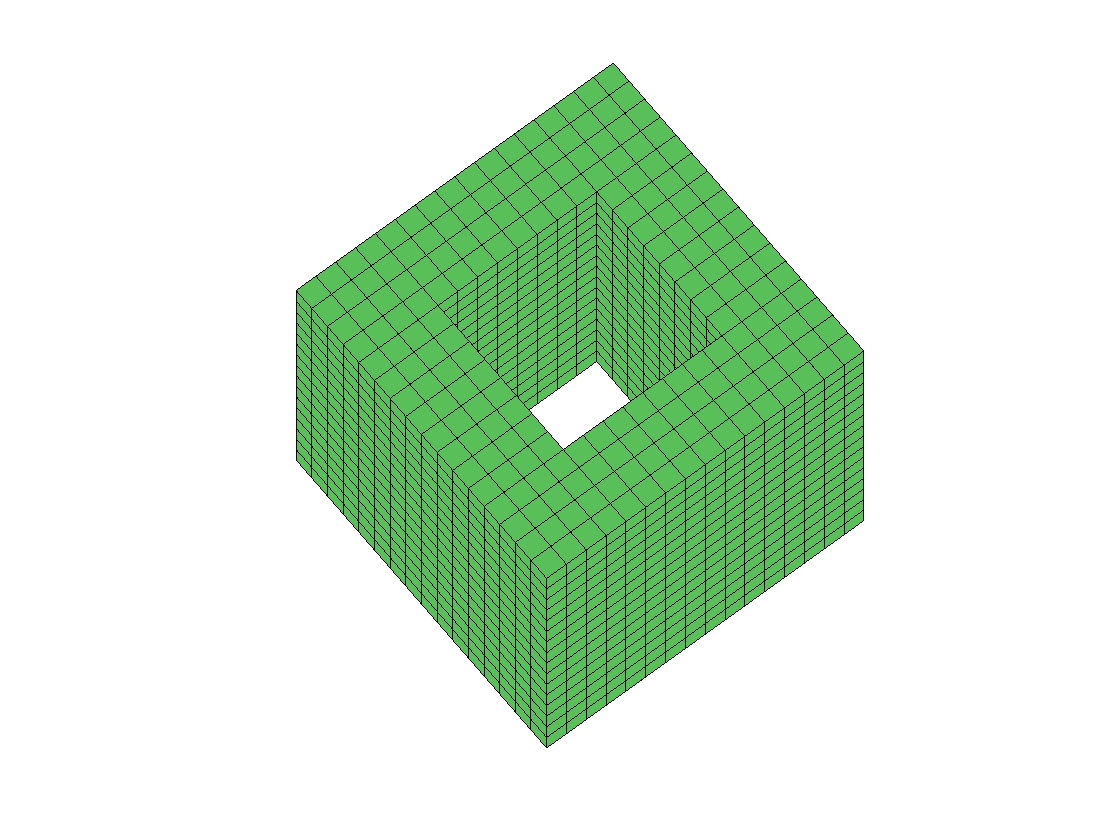}}
\subfigure[]{\label{Fig.sub3.2}
\includegraphics [width=0.3\textwidth]{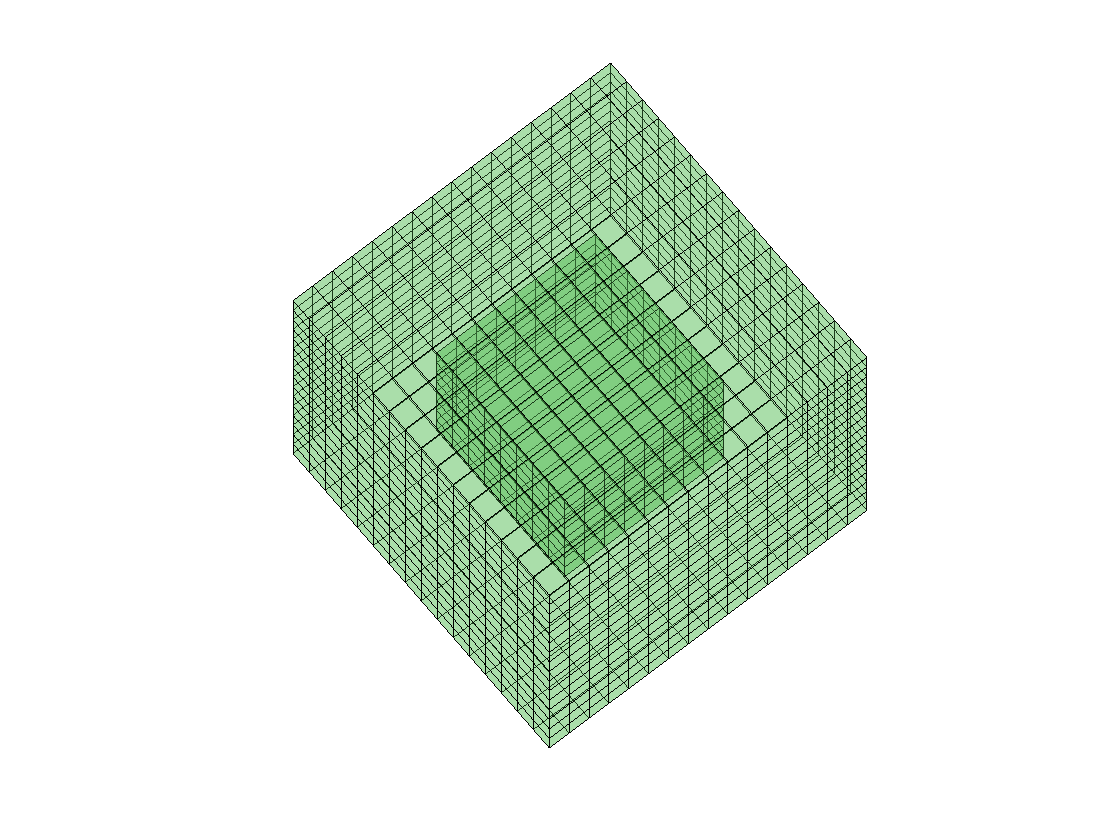}}
\subfigure[]{\label{Fig.sub3.3}
\includegraphics [width=0.3\textwidth]{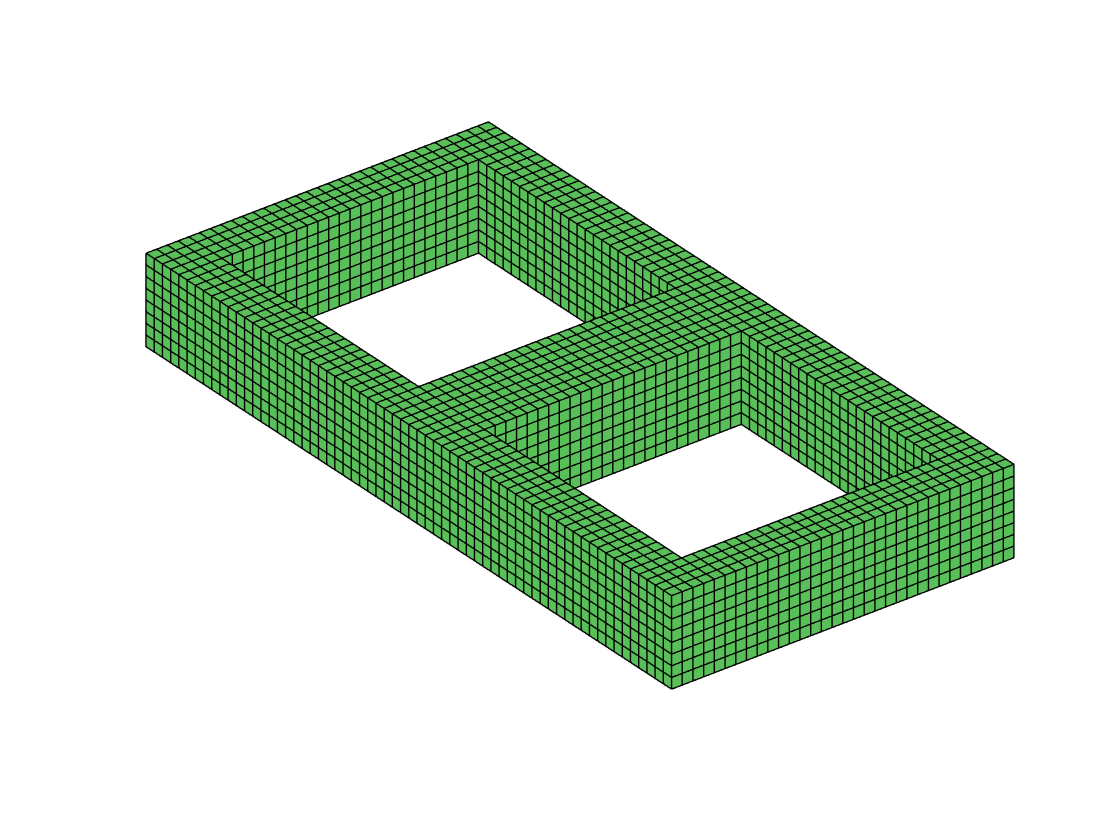}}
\end{center}
\caption{\label{Hole-domain} Multi-connected domains.}
\end{figure}

\begin{table}[!h]
\begin{center}
\caption{Numerical error and convergence performance of the PDWG scheme  for the div-curl system on cubic partitions. $r$ refers to the order of convergence in $O(h^r).$ }\label{table2.1}

\begin{tabular}{||c|cc|cc||}
\hline
\multicolumn{5}{|c|}{ domain a}\\
\hline
\multicolumn{3}{|>{\columncolor{mypink}}c|}{ $\bm{u_5}$ }&\multicolumn{2}{|>{\columncolor{green!15}}c|}{$\bm{u_6}$}\\
\hline
$n$ & $\|e_{\bm{u}} \|_{0}$ & $r=$ & $\|e_{\bm{u}}\|_{0}$ & $r=$   \\
\hline
2  &   1.13e-01 &  -    &    3.05e-01 &     \\
4  &   4.34e-02 &  1.38 &    1.19e-01 &  1.36\\
8  &   1.43e-02 &  1.60 &    4.20e-02 &  1.50\\
\hline
\hline
\multicolumn{5}{|c|}{ domain b}\\
\hline
\multicolumn{3}{|>{\columncolor{yellow!20}}c|}{ $\bm{u_5}$ }&\multicolumn{2}{|>{\columncolor{blue!20}}c|}{$\bm{u_6}$}\\
\hline
$n$ & $\|e_{\bm{u}} \|_{0}$ & $r=$ & $\|e_{\bm{u}}\|_{0}$ & $r=$ \\
\hline
2  &   1.14e-01 &  -     &   2.98e-01  & -     \\
4  &   3.97e-02 &  1.52  &   1.10e-01  & 1.44  \\
8  &   1.23e-02 &  1.69  &   3.68e-02  & 1.58  \\
\hline
\hline
\multicolumn{5}{|c|}{ domain c}\\
\hline
\multicolumn{3}{|>{\columncolor{yellow!20}}c|}{ $\bm{u_5}$ }&\multicolumn{2}{|>{\columncolor{blue!20}}c|}{$\bm{u_6}$}\\
\hline
$n$ & $\|e_{\bm{u}} \|_{0}$ & $r=$ & $\|e_{\bm{u}}\|_{0}$ & $r=$ \\
\hline
2  &   4.13e-02 &  -     &   3.27e-01  & -    \\
4  &   2.16e-02 &  0.94  &   1.87e-01  & 0.81 \\
8  &   7.93e-03 &  1.44  &   7.02e-02  & 1.42 \\
\hline
\end{tabular}
\end{center}
\end{table}
The numerical errors and convergence rates for the scheme \eqref{EQ:PDWG-3d:01:tangential} are reported in Table \ref{table2.1}. It can be seen that the numerical convergence has rates higher than the optimal rate of convergence of $r=1$ in all test cases. The computation thus outperforms the theory for the PDWG scheme \eqref{EQ:PDWG-3d:01:tangential}.
The numerical solutions are plotted in Figure \ref{Num-vec} for each test, which clearly indicate an excellent performance of the PDWG scheme \eqref{EQ:PDWG-3d:01:tangential}.

\begin{figure}[h!]
\begin{center}
\subfigure[Numerical vector field $\bm{u_5}$ on domain a (left) and  b (right) ]{\label{Fig.sub4.l}
\includegraphics [width=0.47\textwidth]{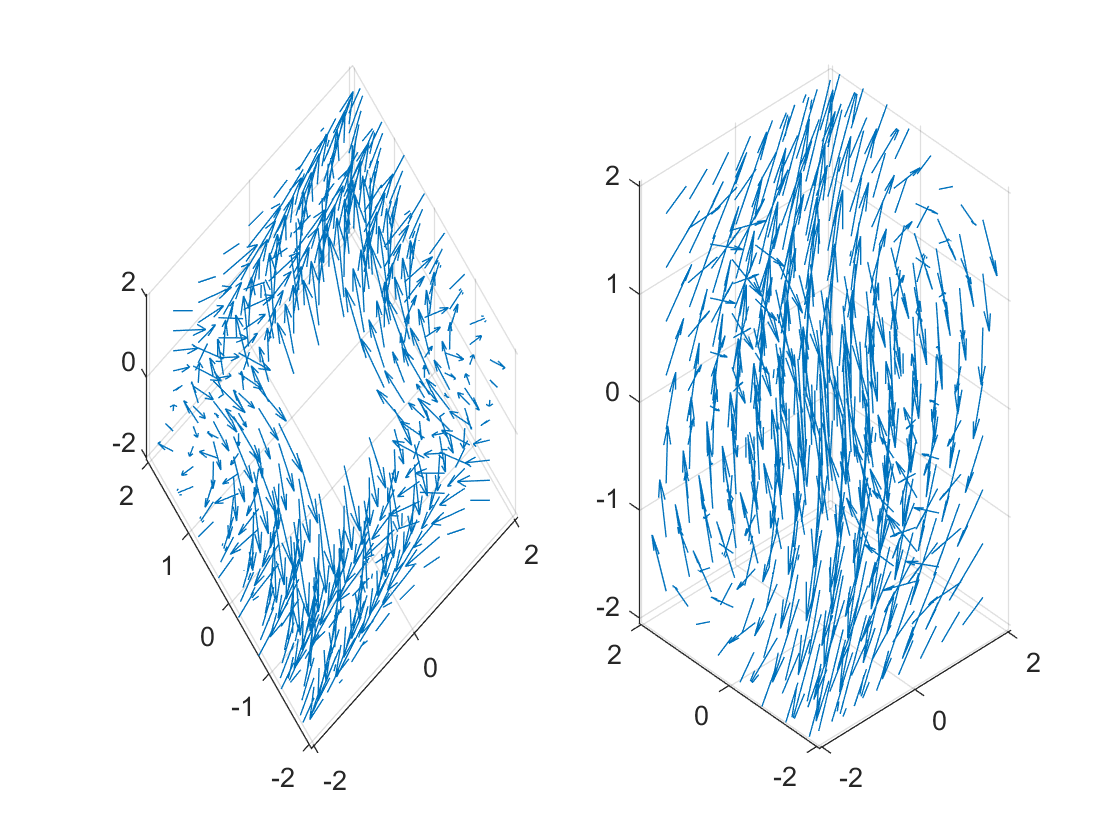}
\includegraphics [width=0.47\textwidth]{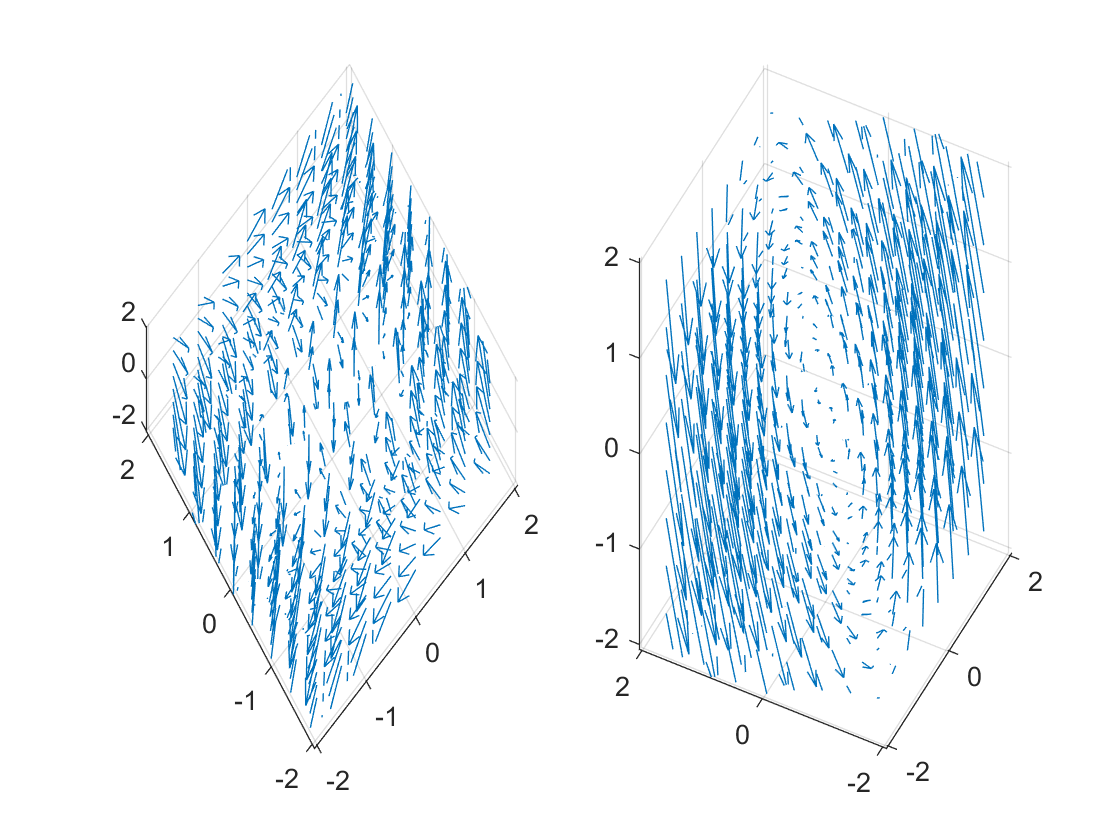}
}
\subfigure[Numerical vector field $\bm{u_6}$ on domain a (left) and  b (right)]{\label{Fig.sub4.2}
\includegraphics [width=0.47\textwidth]{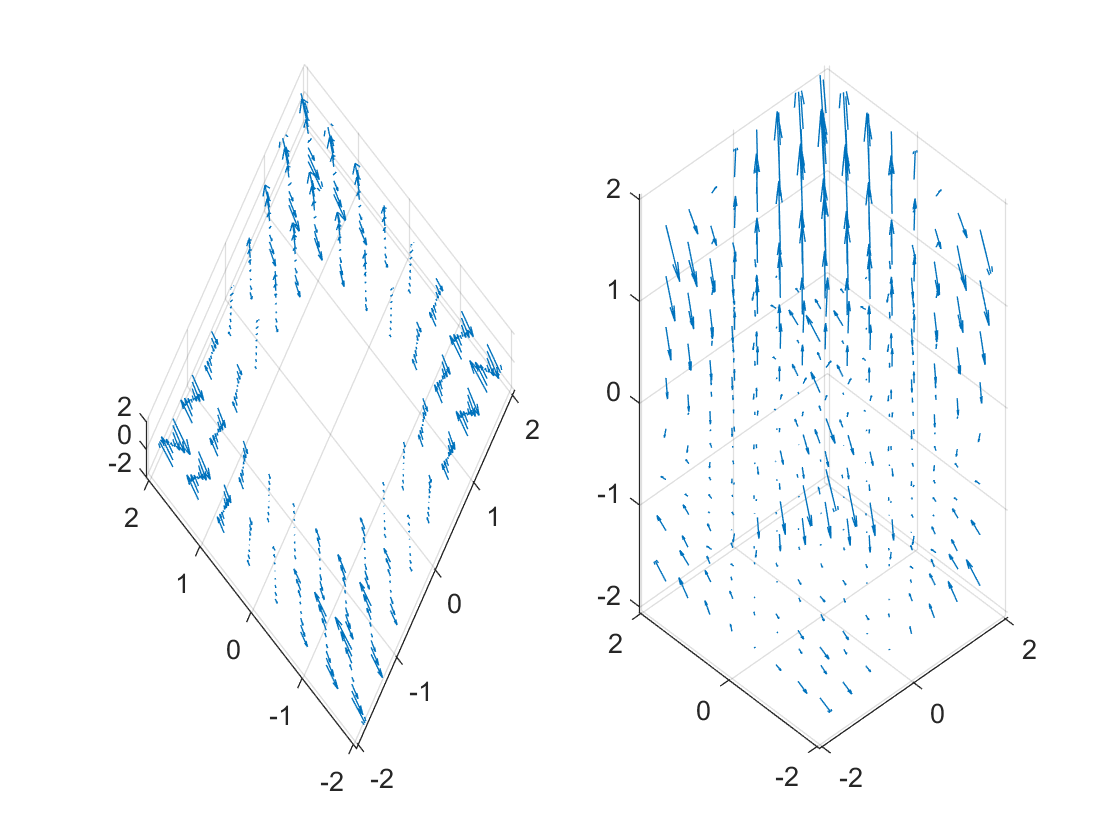}
\includegraphics [width=0.47\textwidth]{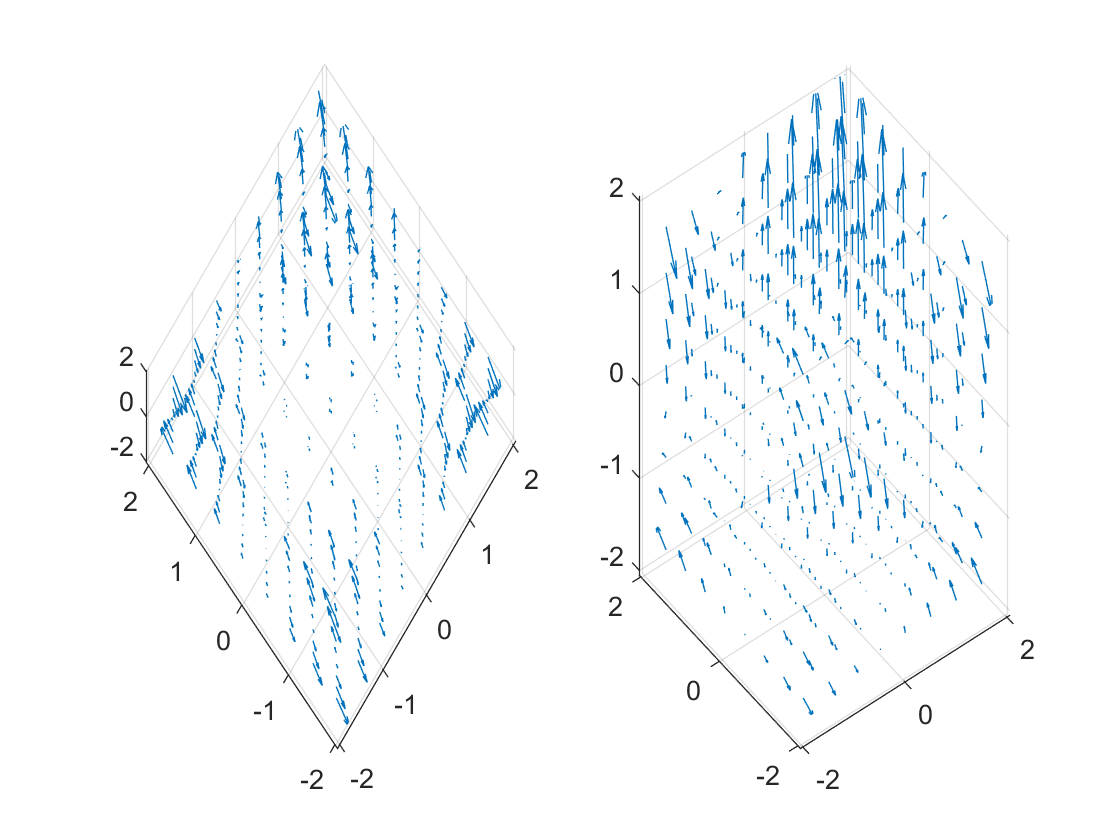}
}
\subfigure[Numerical vector field $\bm{u_5}$ on domain c]{\label{Fig.sub4.3}
\includegraphics [width=0.47\textwidth]{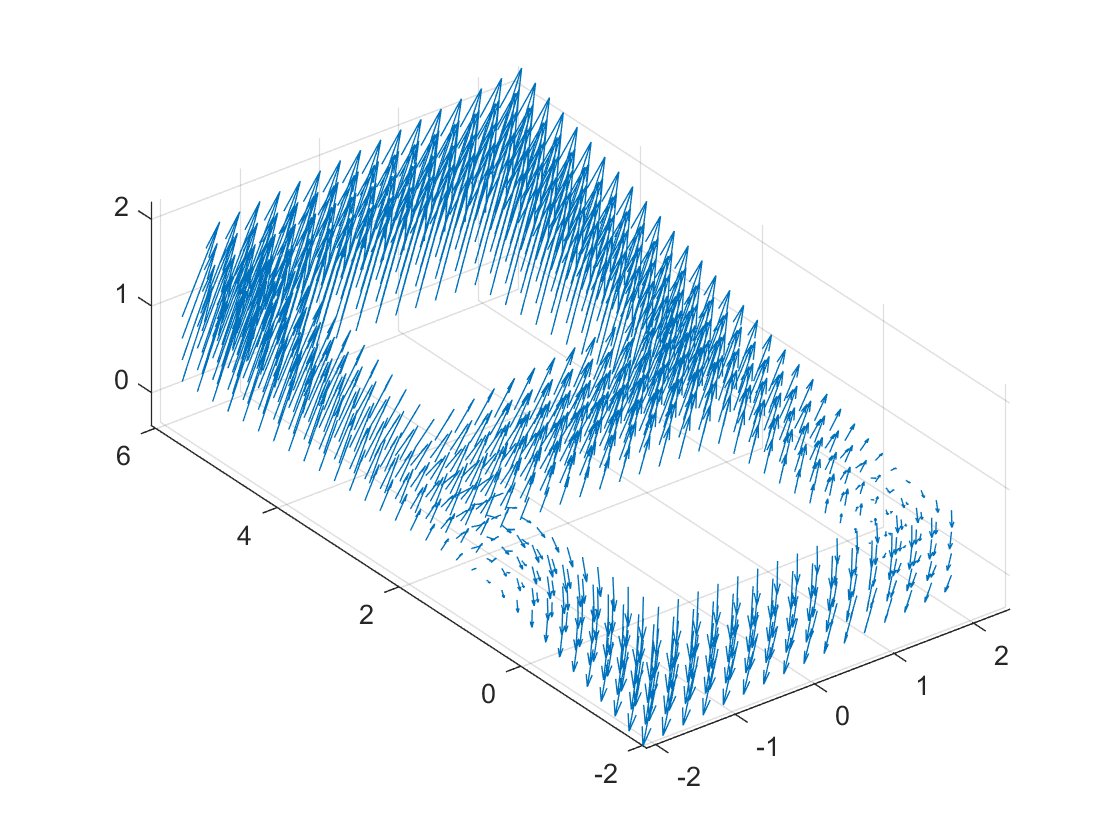}}
\subfigure[Numerical vector field $\bm{u_6}$ on domain c]{\label{Fig.sub4.4}
\includegraphics [width=0.47\textwidth]{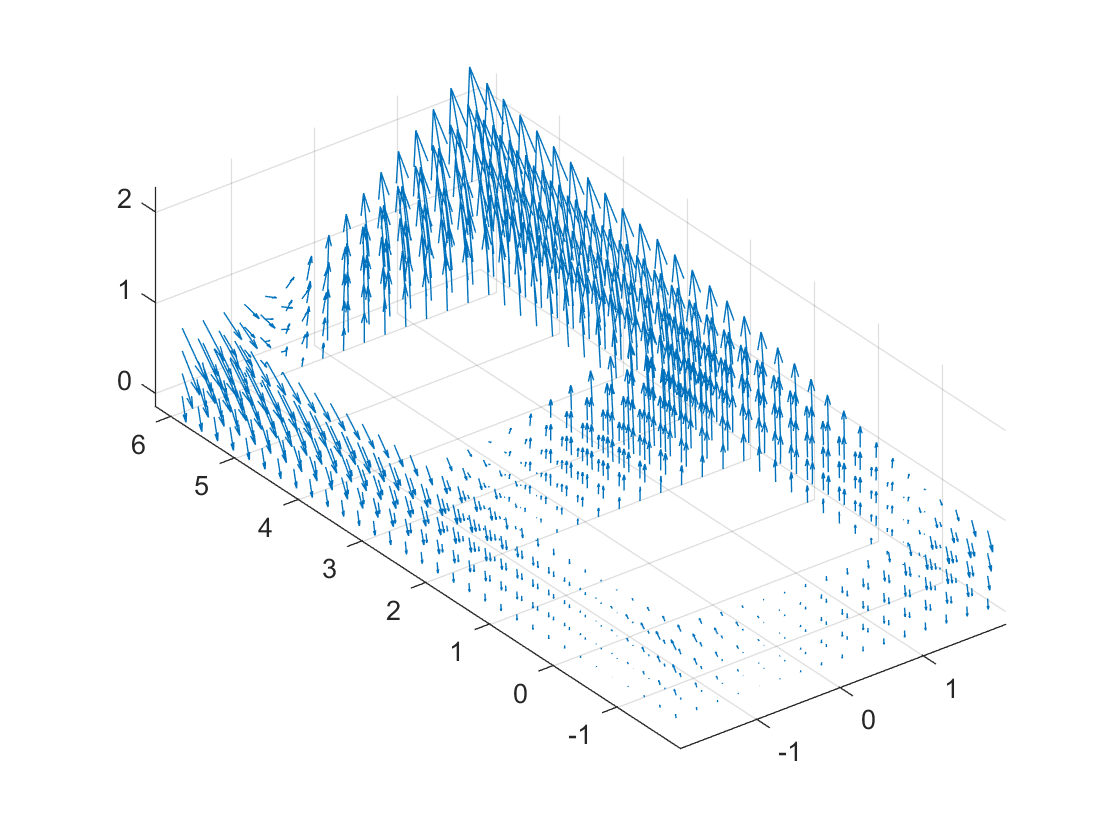}}
\end{center}
\caption{\label{Num-vec} PDWG approximate vector fields for $\bm{u_5}$ and $\bm{u_6}$ on various test domains.}
\end{figure}

\appendix

\section{Helmholtz Decomposition}
The following Helmholtz decomposition holds the key to the
derivation of a suitable variational form for the div-curl problem
\eqref{EQ:div-curl-1}-\eqref{EQ:div-curl-TBC2}.

\begin{theorem}\label{THM:helmholtz-2:02}
For any vector-valued function $\bu\in [L^2(\Omega)]^3$, there
exists a unique $\phi\in H_{0c}^1(\Omega)$ and a vector field $\bpsi\in H(curl;\Omega)$ such that
\begin{equation}\label{EQ:helmholtz-2:02}
\bu = \varepsilon^{-1}\nabla\times\bpsi + \nabla\phi,
\end{equation}
where $\bpsi$ additionally satisfies
\begin{equation}\label{EQ:additionalCond}
\nabla\cdot\bpsi=0,\quad \bpsi\cdot\bn =0 \quad \mbox{on } \partial\Omega.
\end{equation}
\end{theorem}

\begin{proof}
A proof of the decomposition \eqref{EQ:helmholtz-2:02} has been given in \cite{WangWang_divcurl_2016}. Below is an outline of the proof.

For any vector field $\bu\in [L^2(\Omega)]^3$, let $\phi\in H_{0c}^1(\Omega)$ be the unique solution of the following problem:
$$
(\varepsilon \nabla\phi, \nabla s) = (\varepsilon\bu, \nabla s)\quad \forall\; s\in H_{0c}^1(\Omega).
$$
By letting $\bv=\bu-\nabla\phi$, it is not hard to see that
$$
\nabla\cdot(\varepsilon\bv)=0,\quad \langle \varepsilon\bv\cdot\bn, 1\rangle_{\Gamma_i}=0
$$
for $i=0,1,\cdots,L$. Thus, from
Theorem 3.4 of \cite{girault-raviart}, there exists a vector
potential \textcolor{blue}{field} $\bpsi\in [H^1(\Omega)]^3$ such that
\begin{eqnarray}\label{ForCM-02:01}
&&\varepsilon\bv = \nabla\times\bpsi,\ \nabla\cdot\bpsi=0,\\
&&\|\bpsi\|_1 \leqC  (\varepsilon\bv,\bv)^{\frac12}.\label{ForCM-02:02}
\end{eqnarray}
Furthermore, from Theorem 3.5 of \cite{girault-raviart}, among all the vector fields $\bpsi$ satisfying \eqref{ForCM-02:01}, we may choose $\bpsi\in H(curl;\Omega)$ such that
$$
\bpsi\cdot\bn = 0\qquad\mbox{on }\partial\Omega.
$$
\end{proof}
\medskip

It should be pointed out that the vector field $\bpsi$ in the Helmholtz decomposition \eqref{EQ:helmholtz-2:02} is not uniquely determined by the condition \eqref{EQ:additionalCond}, as nothing will change when $\bpsi$ is altered by any harmonic function $\mathbb{H}_{n,0}(\Omega)$. But the decomposition \eqref{EQ:helmholtz-2:02} would be unique when $\bpsi$ is restricted to the $L^2$-orthogonal complement of the harmonic space $\mathbb{H}_{n,0}(\Omega)$.

\clearpage


\begin{thebibliography}{99}
\bibitem{giles}
{G. Auchmuty and J.C. Alexander},
$L^2$ Well-posedness of planar div-curl systems, {\em Archive for
Rational Mechanics and Analysis}, 20 (2001), 160(2), 91-134.

\bibitem{Babuska1973}
I. Babu$\breve{s}$ka,
\newblock{ The finie element method with Lagrange multipliers},
\newblock {\em Numer. Math.}, 20 : 173-192, 1973.

\bibitem{Bensow2005}
R. Bensow and M.G. Larson,
\newblock{Discontinuous least-squares finite element method for the div-curl problem},
\newblock {\em Numerische Mathematik}, 101.4 (2005): 601-617.

\bibitem{BochevPetersonSiefert_2011}
P.B. Bochev, K. Peterson, and C.M. Siefert,
Analysis and computation of compatible least-squares methods for div-curl equations,
\newblock {\em SIAM J. Numer. Anal.} 2011, 49: 159-181.

\bibitem{Bossavit_1998}
A. Bossavit,
\newblock{Computational Electromagnetism},
\newblock {\em Academic Press Inc.}, San Diego, CA, 1998.

\bibitem{BramblePasciak_2003}
J.H. Bramble and J.E. Pasciak,
\newblock{A new approximation technique for div-curl systems},
\newblock {\em Math. Comp.}, 2003, 73: 1739-1762.

\bibitem{Brezzi1974}
F. Brezzi,
\newblock {On the existence, uniqueness, and approximation of saddle point problems arising from Lagrange multipliers,}
\newblock{\em RAIRO}, 8: 129-151, 1974.

\bibitem{BrezziBuffa_2010}
R. Brezzi and A. Buffa,
Innovative mimetic discretizations for electromagnetic problems, J.
\newblock{\em Comput. Appl. Math.}, 2010, 234: 1980-1987.

\bibitem{burman2013}
{\sc E. Burman}, {Stabilized finite element methods for nonsymmetric, noncoercive, and ill-posed problems. Part I: Elliptic equations}, {\em SIAM J. Sci. Comput.}  35(6)  (2013),  A2752-A2780.

\bibitem{burman2014}
{\sc E. Burman},
Stabilized finite element methods for nonsymmetric, noncoercive, and ill-possed problems. Part II: hyperbolic equations,
{\em SIAM J. Sci. Comput}, vol. 36, No. 4, pp. A1911-A1936, 2014.

\bibitem{Ciarlet_1978}
P.G. Ciarlet,
\newblock The Finite Element Method for Elliptic Problems,
{\it North-Holland}, 1978.

\bibitem{CopelandGopalakrishnanPasciak_2008}
D.M. Copeland, J. Gopalakrishnan, and J.E. Pasciak,
A mixed method for axisymmetric div-curl systems,
{\it Math. Comput.}  2008, 77: 1941-1965.

\bibitem{DelcourteDomelevoOmnes_2007}
S. Delcourte, K. Domelevo, and P. Omnes,
\newblock A discrete duality finite volume approach to Hodge decomposition and div-curl problems on almost arbitrary two-dimensional meshes,
{\it SIAM J. Numerical Analysis}, 45 (2007), pp. 1142-1174.

\bibitem{girault-raviart}
{\sc V. Girault and P-A Raviart}, \textit{Finite Element Methods for
Navier-Stokes Equations: Theory and Algorithms}, Springer-Verlag
Berlin Heidelberg, 1986.

\bibitem{LiYeZhang_2018}
\newblock {J. Li, X. Ye, and S. Zhang,}
A weak Galerkin least-squares finite element method for div-curl systems.
\newblock {\em Journal of Computational Physics}, 2018, 363: 79-86.

\bibitem{LipnikovManziniBrezzi_2011}
K. Lipnikov, G. Manzini, F. Brezzi, and A. Buffa,
The mimetic finite difference method for the 3D magnetostatic field problems on polyhedral meshes,
\newblock {\em J. Comput. Phys.} 2011, 230: 305-328.

\bibitem{LiuWang_SINUM_2017}
\text{Y. Liu, J. Wang, and Q. Zou,}
\newblock  A conservative flux-optimization finite element method for convection-diffusion equations,
\newblock {\em SIAM J. Numer. Anal.} 2019, 57(3): 1238-1262.

\bibitem{Nicolaides_1992}
R.A. Nicolaides,
Direct discretization of planar div-curl problems,
\newblock {\em SIAM J. Numer. Anal.}, 1992, 29: 32-56.

\bibitem{Nicolaides_1997}
R.A. Nicolaides and X. Wu,
Covolume solutions of three-dimensional div-curl equations,
\newblock {\em  SIAM J. Numer. Anal.}, 1997, 34 : 2195-2203.

\bibitem{YeZhangZhu_2020}
X. Ye, S. Zhang, and P. Zhu,
A discontinuous Galerkin least-squares method for div-curl systems.
\newblock {\em Journal of Computational and Applied Mathematics}, 2020, 367: 112474

\bibitem{RanacherScott1982} R. Rannacher  and R. Scott.
Some optimal error estimate for piecewise linear finite element approximations,
{\em Math. Comp.}, 38: 437-445, 1982.

\bibitem{WangWang_divcurl_2016}
C. Wang and J. Wang,
\newblock Discretization of div-curl systems by weak Galerkin finite element methods on polyhedral partitions.
 {\it Journal of Scientific Computing,} 2016, 68(3): 1144-1171.

\bibitem{WangWang_2016}
C. Wang and J. Wang,
\newblock{A primal-dual weak Galerkin finite element method for second elliptic equations in non-divergence form},
{\it Math. Comp.}, 2018, 87(310): 515-545.

\bibitem{WangWang_2017}
C. Wang and J. Wang,
\newblock{A primal-dual weak Galerkin finite element method for Fokker-Planck type equations},
arXiv:1704.05606. {\it SIAM J. Numer. Anal.}, 2020, 58(5), 2632-2661.

\bibitem{ww2020-ec} {\sc C. Wang and J. Wang}, {Primal-dual weak Galerkin finite element methods for elliptic Cauchy problems}, {\em Computers and Mathematics with Applications}, vol 79(3), pp. 746-763, 2020.

\bibitem{ww2020-transport} {\sc C. Wang and J. Wang}, {A primal-dual finite element method for first-order transport problems}, {\em Journal of Computational Physics}, Vol. 417, 109571, 2020.

\bibitem{WangYe_2013}
J. Wang and X. Ye.
\newblock{\textcolor{blue}{A weak Galerkin finite element method for second-order elliptic problems.}}
available at arXiv: 1104.2897vl.
{\it J. Comp. and Appl. Math.}, {\bf 241}, 103-115, 2013.

\bibitem{wy3655}
{J. Wang and X. Ye}. A weak Galerkin mixed finite element
method for second-order elliptic problems, {\em Math. Comp.}, 83 (2014), pp. 2101-2126.


\end{thebibliography}
\end{document}